\documentclass[11pt]{amsart}

\usepackage{amsmath}
\usepackage{amssymb}
\usepackage{bbm}
\usepackage{pdfsync}
\usepackage{esint} 

\usepackage{graphicx,psfrag}
\usepackage{pstricks,pst-node,pst-text,pst-3d}
\usepackage{tikz}

\definecolor{almond}{rgb}{0.94, 0.87, 0.8}

\newcommand{\R}{{\mathbb R}}       

\newcommand{\E}{{\mathbb E}}

\newcommand{\Z}{{\mathbb Z}}       
\newcommand{\DD}{{\mathcal D}}

\newcommand{\HH}{{\mathcal H}}
\newcommand{\II}{{\mathcal I}}

\newcommand{\WW}{{\mathcal W}}

\newcommand{\GZ}{{\mathcal G}}

\newcommand{\TT}{{\mathcal T}}

\newcommand{\EE}{{\mathcal E}}

\newcommand{\CC}{{\mathcal C}}

\newcommand{\diam}{{\rm diam}}
\newcommand{\dist}{{\rm dist}}

\newcommand{\fiproof}{{\hspace*{\fill} $\square$ \vspace{2pt}}}

\newcommand{\rf}[1]{{(\ref{#1})}}

\newcommand{\vphi}{{\varphi}}
\newcommand{\ve}{{\varepsilon}}
\newcommand{\vv}{{\vspace{2mm}}}
\newcommand{\vvv}{{\vspace{3mm}}}
\newcommand{\wt}[1]{{\widetilde{#1}}}

\def\Xint#1{\mathchoice
{\XXint\displaystyle\textstyle{#1}}%
{\XXint\textstyle\scriptstyle{#1}}%
{\XXint\scriptstyle\scriptscriptstyle{#1}}%
{\XXint\scriptscriptstyle\scriptscriptstyle{#1}}%
\!\int}
\def\XXint#1#2#3{{\setbox0=\hbox{$#1{#2#3}{\int}$ }
\vcenter{\hbox{$#2#3$ }}\kern-.58\wd0}}

\def\avint{\;\Xint-}

\usepackage{pgf,tikz} 

\definecolor{ffffff}{rgb}{1.0,1.0,1.0}
\definecolor{qqqqff}{rgb}{0.0,0.0,1.0}
\definecolor{ffqqqq}{rgb}{1.0,0.0,0.0}
\definecolor{zzzzqq}{rgb}{0.6,0.6,0.0}
\definecolor{marronet}{rgb}{0.6,0.2,0}
\definecolor{negre}{rgb}{0,0,0}
\definecolor{vermell}{rgb}{0.8,0.05,0.05}
\definecolor{blau}{rgb}{0.3,0.2,1.}
\definecolor{blauclar}{rgb}{0.,0.,1.}
\definecolor{grisfosc}{rgb}{0.25098039215686274,0.25098039215686274,0.25098039215686274}
\definecolor{verd}{rgb}{0.1,0.6,0.1}
\definecolor{taronja}{rgb}{0.9,0.6,0.05}
\definecolor{vermellclar}{rgb}{1.,0.,0.}
\definecolor{verdet}{rgb}{0,0.8,0.1}
\definecolor{blauverd}{rgb}{0,0.4,0.2}
\definecolor{grisclar}{rgb}{0.6274509803921569,0.6274509803921569,0.6274509803921569}

\newtheorem{theorem}{Theorem}[section]
\newtheorem{lemma}[theorem]{Lemma}

\newtheorem{keylemma}[theorem]{Key Lemma}
\newtheorem{coro}[theorem]{Corollary}

\newtheorem*{claim*}{Claim}
\newtheorem*{conjecture}{Conjecture}
\newtheorem*{theorem*}{Theorem}

\theoremstyle{definition}

\theoremstyle{remark}
\newtheorem{rem}[theorem]{\bf Remark}

\numberwithin{equation}{section}

\newcommand{\brem}{\begin{rem}}
\newcommand{\erem}{\end{rem}}

\begin{document}

\title[Unique continuation at the boundary]{Unique continuation at the boundary for harmonic functions in $C^1$ domains and Lipschitz domains with small constant}

\author{Xavier Tolsa}

\address{Xavier Tolsa \\
ICREA, Passeig Llu\'{\i}s Companys 23 08010 Barcelona, Catalonia;  Departament de Matem\`atiques, 
Universitat Aut\`onoma de Barcelona, 08193 Bellaterra, Catalonia; and Centre de Recerca Matem\`atica,
08193 Bellaterra, Catalonia.
}


\thanks{Supported by 2017-SGR-0395 (Catalonia) and MTM-2016-77635-P (MINECO, Spain).}


\begin{abstract}
Let $\Omega\subset\R^n$ be a $C^1$ domain, or more generally, a 
Lipschitz domain with small local Lipschitz constant. In this paper it is shown that if $u$ is a
function  harmonic in $\Omega$ and continuous in $\overline \Omega$ 
which vanishes in a relatively open subset $\Sigma\subset\partial\Omega$ and moreover
the normal derivative $\partial_\nu u$ vanishes in a subset of $\Sigma$ with positive surface
measure, then $u$ is identically zero.
\end{abstract}

\maketitle

\section{Introduction}

In $\R^n$, with $n\geq3$, there are examples of harmonic functions in the half-space $\R^{n}_+$, 
$C^1$ up to the boundary, such that the function and its normal derivative vanish simultaneously on a 
set of positive measure of $\partial \R^{n}_+$. This was shown by Bourgain and Wolff in \cite{BW}.
The same result was generalized later to arbitrary $C^{1,\alpha}$ domains by Wang \cite{Wang}.
A related conjecture which is still open is the following:

\begin{conjecture}
Let $\Omega\subset\R^n$ be a Lipschitz domain and let $\Sigma\subset \partial\Omega$ be relatively open with respect to $\partial\Omega$. Let $u$ be a function harmonic in $\Omega$ and continuous in $\overline\Omega$. Suppose that $u$  vanishes in $\Sigma$ and the normal derivative $\partial_\nu u$ vanishes
in a subset of $\Sigma$ with positive surface measure. Then $u\equiv0$ in $\overline\Omega$.
\end{conjecture}

Remark that the assumption that
$u$ vanishes continuously in $\Sigma$ implies that $\nabla u$ exists $\sigma$-a.e.\ as a non-tangential
limit in $\Sigma$, and moreover $\nabla u = (\partial_\nu u)\,\nu \in L^2_{loc}(\sigma|_\Sigma)$. Here
 $\sigma$ stands for the  $(n-1)$-dimensional surface measure and $\nu$ is the outer unit normal. See Appendix A for more details.
 
The preceding conjecture is an open problem which is already mentioned in Fang-Hua Lin's work \cite{Lin}. It was later stated explicitly as a conjecture in the works by Adolfsson, Escauriaza, and Kenig
\cite{AEK}, \cite{AE}\footnote{
For an accurate historical account, see the recent work \cite{Kenig-Zhao}.}.
 The conjecture is known to be true in the plane, and also in higher dimensions if one assumes the function $u$ to be positive. In the first case, this can be deduced from the subharmonicity of $\log|\nabla u|$, and in the second one it is possible to use standard techniques in connection with harmonic measure and the comparison principle. 

In this paper I show that the conjecture is true for Lipschitz domains with small local Lipschitz constant.
The precise result is the following:

\begin{theorem}\label{teomain}
Let $\Omega\subset\R^n$ be a Lipschitz domain, let $B$ be a ball centered in $\partial\Omega$,
and suppose that $\Sigma = B\cap \partial\Omega$ is a Lipschitz graph with slope at most $\tau_0$,
where $\tau_0$ is some positive small enough constant depending only on $n$. 
Let $u$ be a function harmonic in $\Omega$ and continuous in $\overline\Omega$. Suppose that $u$  vanishes in $\Sigma$ and the normal derivative $\partial_\nu u$ vanishes
in a subset of $\Sigma$ with positive surface measure. Then $u\equiv0$ in $\overline\Omega$.
\end{theorem}

As an immediate corollary, it follows that the above conjecture holds for $C^1$ domains.

Remark that, up now, the result stated in Theorem \ref{teomain} (and in the conjecture) was only known in the
case of Dini domains (i.e., Lipschitz domains whose outer normal is Dini continuous), by results of
 Adolfsson and Escauriaza \cite{AE} and Kukavica and Nystr\"om \cite{KN}, and also in the
case of convex Lipschitz domains, by Adolfsson, Escauriaza, and Kenig \cite{AEK}. 
Previously, the case of $C^{1,1}$ domains had been solved by F.-H. Lin \cite{Lin}. See also \cite{MC}
for a recent contribution in the particular case of convex domains where the 
recent geometric techniques introduced by Naber and Valtorta \cite{NV} are applied to study the strata of the set where $\partial_\nu u$ vanishes.

The proof of Theorem \ref{teomain} is based on the study of the doubling properties of $L^2$ averages of the harmonic
function $u$ by means of the so called Almgren's frequency function,
analogously to the  works mentioned in the previous paragraph.  The strategy in this paper 
consists in studying the behavior of the frequency 
function at points in $\Omega$ approaching the boundary. This strategy
is closer to the one of 
Kukavica and Nystr\"om in \cite{KN} than to the one of Adolfsson and Escauriaza \cite{AE}, which is
based on the use of a clever change of variables that transforms the Laplace equation
into an elliptic PDE in divergence form with non-constant coefficients and improves the domain, in a sense.

The main novelty in the arguments to prove Theorem \ref{teomain} is the application of some combinatorial techniques developed by Logunov and Malinnikova
in the works \cite{Logunov1}, \cite{Logunov2}, \cite{LM} in connection with the nodal sets of harmonic functions and
the Nadirashvili and Yau conjectures. In particular, one of the main technical results in this paper,
the Key Lemma \ref{keylemma} uses some ideas inspired by \cite{Logunov1} to bound the set where the
frequency function is large.
With the Key Lemma \ref{keylemma} in hand, in the last section of the paper a probabilistic argument is used
to show that
$$\liminf_{r\to0} \frac{\int_{\partial B(x,12r)} u^2\,d\sigma}{\int_{\partial B(x,r)} u^2\,d\sigma} < \infty$$
for almost all points  $x\in\Sigma$. By a lemma due to Adolfsson and Escauriaza \cite[Lemma 0.2]{AE},
this suffices to show that $\partial_\nu u$ cannot vanish
in a subset of $\Sigma$ with positive measure.

In case that $\Omega$ is a Dini domain, in \cite{AE} and \cite{KN} it is also proven that if $u$ is harmonic in $\Omega$ and vanishes continuously 
in $\Sigma$ (where $\Sigma$ is as in Theorem \ref{teomain}), then $|\partial_\nu u| $ is a local $B_2$ weight
in $\Sigma$ with respect to surface measure (i.e., it satisfies a local reverse H\"older inequality with exponent $2$). This follows
from the local uniform bound of Almgren's frequency function proven in \cite{AE} and \cite{KN}, which in turn implies 
a local uniform doubling condition for $L^2$ averages of the function $u$ on surface balls. Then, as
shown in \cite[Theorem 1]{AEK}, this doubling condition ensures that $|\partial_\nu u| $ is a local $B_2$ weight. In the case of Lipschitz domains with small constant, the proof of Theorem \ref{teomain} in this paper does
not ensure that the frequency function is locally uniformly bounded (or even pointwise bounded!) in $\Sigma$, and
thus one cannot deduce that $|\partial_\nu u| $ is a local $B_2$ weight.

In a similar vein, under the Dini assumption, in \cite{AE} it is shown that the dimension of the set
where $\partial_\nu u$ vanishes in $\Sigma$ has dimension at most $n-2$. This follows by arguments developed
previously in \cite{Lin} in the case of $C^{1,1}$ domains, which are based on the monotonicity of the 
frequency function in $\Sigma$. For $C^1$ domains or Lipschitz domains with small constant one cannot
derive any bound on the Hausdorff dimension 
smaller than $n-1$ from the arguments in this paper, as far as I know.

Finally it is worth mentioning a corollary about harmonic measure that follows easily from Theorem \ref{teomain}:

\begin{coro}
Let $\Omega\subset\R^n$ be a Lipschitz domain, let $B$ be a ball centered in $\partial\Omega$, and suppose that $\Sigma=B\cap \partial\Omega$
is a Lipschitz graph with small enough slope.
Let $\omega^p$, $\omega^q$ be the harmonic measures for $\Omega$ with respective poles in $p,q\in\Omega$.
Suppose that there exists some subset $E\subset\Sigma$ with positive harmonic measure such that
$$\omega^p|_E = \omega^q|_E$$
Then \,
$p = q.$
\end{coro}

Remark that saying that $\omega^p|_E = \omega^q|_E$ is the same as saying that 
$$\omega^p(F) = \omega^q(F)\quad \mbox{ for all Borel sets $F\subset E$.}$$
 The corollary follows by applying the theorem to 
$$u=  g(\cdot,p) - g(\cdot,q)\quad \mbox{ in $\Omega\setminus (\overline B(p,\ve)\cup \overline B(q,\ve))$,}$$
where $g(\cdot,\cdot)$ is the Green function of $\Omega$ and $\ve>0$ is small enough so that
$B(p,2\ve)\cup B(q,2\ve)\subset \Omega$. Observe that $u$ is harmonic in the Lipschitz domain
$\Omega_\ve:=\Omega\setminus (\overline B(p,\ve)\cup \overline B(q,\ve))$, it is continuous in $\overline{
\Omega_\ve}$, and it vanishes identically in $\partial\Omega \subset \partial\Omega_\ve$. Further, by Dahlberg's classical
theorem \cite{Dah}, it follows that the harmonic measures $\omega^p$ and $\omega^q$ are mutually absolutely 
continuous with the surface measure $\sigma$ on $\partial\Omega$, with 
$$\frac{d\omega^p}{d\sigma} = -\partial_\nu g(\cdot,p),\qquad \frac{d\omega^q}{d\sigma} = -\partial_\nu g(\cdot,q),$$
so that $\partial_\nu u$ vanishes in the whole $E$. So $\Omega_\ve$ and $u$ satisfy the assumptions
of Theorem \ref{teomain}, and thus $u\equiv 0$ in $\Omega_\ve$. This implies that $p=q$. Otherwise, letting $\ve\to 0$ we infer that $u(x)\to\infty$ as $x\to p$.

\vv

I would like to thank Luis Escauriaza and Albert Mas for useful discussions regarding the topic of the paper, and also Luis Escauriaza for
some suggestions that improved the readability of the paper.

\vv


\section{The frequency function}

As usual in harmonic analysis, in the whole paper, the letters $C,c$ are used to denote positive constants
which just depend on the dimension $n$ and whose values may change at different occurrences.
On the other hand, constants with subscripts, such as at $C_0$, retain their values in different occurrences. The notation $A\lesssim B$ is equivalent to $A\leq C \,B$, and $A\approx B$ is
equivalent to $A\lesssim B\lesssim A$.
\vv

In the whole paper, unless otherwise stated, we assume that $\Omega$ and $\Sigma$ are as in 
Theorem \ref{teomain}. 
We consider a function $u$ harmonic in $\Omega$ and continuous in $\overline \Omega$ which vanishes in $\Sigma$, and we assume that $u$ is not constant in $\Omega$. 
We extend $u$ by $0$
out of $\overline\Omega$, so that $u$ is continuous across $\Sigma$. 
Without loss of generality we assume that $\Sigma$ is a Lipschitz graph with respect the horizontal axes and that $\Omega\cap B$ lies above $\Sigma\cap B$. For $0<\ve\leq \frac12r(B)$, we denote $\Sigma_\ve = \Sigma + \ve\,e_n$
and $\Omega_\ve = \Omega + \ve\,e_n$,
where $e_n=(0,\ldots,0,1)$. 

For $x\in\R^n$ and $r>0$, we denote
$$h(x,r) = \frac1{\sigma(\partial B(x,r))} \int_{\partial B(x,r)} u^2\,d\sigma.$$
For a ball $B(x,r)$ which intersects $\Omega$, the Almgren frequency function (or just frequency function) associated with $u$ is defined by:
$$F(x,r) = 
r\,\partial_r \log h(x,r).$$

\begin{lemma}\label{lem-hxr}
Let $x\in\R^n$ and let $I\subset (0,\infty)$ be a closed bounded interval. Suppose that $B(x,r)\cap\Omega\neq\varnothing$ and $\bar B(x,r)\subset 2B$ for all $r\in I$ (where $B$ is as
Theorem \ref{teomain}). Then
$h(x,\cdot)$ is of class $C^1$ in $I$ and 
\begin{align}\label{eq1**}
\partial_r h(x,r) &= \frac2{\sigma(\partial B_r)} \int_{\partial B(x,r)} u(y)\,\nabla u(y)\cdot \frac{y-x}r\,d\sigma(y)= \frac2{\sigma(\partial B_r)} \int_{\partial B(x,r)\cap \Omega} u\,\partial_\nu  u\,d\sigma\\
& =\frac2{\sigma(\partial B_r)} \int_{B(x,r)\cap \Omega} |\nabla u|^2\,dy\qquad \mbox{ for all $r
\in I.$}\nonumber
\end{align}
Also,
\begin{equation}\label{eq2**}
F(x,r) = r\,\frac{\partial_r h(x,r)}{h(x,r)} =\frac{2 r\int_{B(x,r)} |\nabla u|^2\,dy}
{\int_{\partial B(x,r)} u^2\,d\sigma} \quad \mbox{ for all $r
\in I.$}
\end{equation}
\end{lemma}

Remark that, for $B(x,r)$ as in the lemma, we have
$$\int_{B(x,r)\cap \Omega} |\nabla u|^2\,dy<\infty.$$
Indeed, write $u=u^+ - u^-$ and, for any $\ve>0$, let $u^+_\ve=\max(u^+,\ve)-\ve$,
$u^-_\ve=\max(u^-,\ve)-\ve$, $v_\ve=u_\ve^+-u_\ve^-$. It is immediate to check that $u^+_\ve$ and $u_\ve^-$ belong to $W^{1,2}(B(x,r'))$ for some $r'>r$, and moreover they are subharmonic in
$B(x,r')$. As a consequence,
 by Caccioppoli's inequality,
$$\int_{B(x,r)} |\nabla u^\pm_\ve|^2\,dy \lesssim C(r,r')\int_{B(x,r')} |u^\pm_\ve|^2\,dy\leq C(r,r')\int_{B(x,r')} |u|^2\,dy.$$
Letting $\ve\to0$, we deduce that 
$$\int |\nabla u|^2\,dy<\infty.$$ 
Further, it follows easily that $v_\ve\to u$ in $W^{1,2}(B(x,r))$, and so $u\in W^{1,2}(B(x,r))$ too.

An immediate corollary of the lemma and, in particular, of the third identity in \rf{eq1**} is that $\partial_r h(x,r)\geq0$ and thus $h(x,r)$ is non-decreasing with respect to $r$. 

\begin{proof}[Proof of Lemma \ref{lem-hxr}]
The calculations in the lemma are quite straightforward and well-known in the case when $u$ is
sufficiently smooth up to boundary. In the general case when we only assume $u$ to be continuous
up to the boundary we have to be a little more careful and so we will show here the whole details.

Notice first that the second identity in \rf{eq1**} is immediate. Concerning the third one, for all $r\in I$ we have
$$\frac2{\sigma(\partial B_r)} \int_{\partial B(x,r)\cap \Omega} u\,\partial_\nu  u\,d\sigma
=\lim_{\ve \to 0} \frac2{\sigma(\partial B_r)} \int_{\partial (B(x,r)\cap \Omega_\ve)} u\,\partial_\nu  u\,d\sigma$$
because
$$\lim_{\ve \to 0} \frac2{\sigma(\partial B_r)} \int_{B(x,r)\cap \partial\Omega_\ve} u\,\partial_\nu  u\,d\sigma=0.$$
This follows easily from the fact that $u$ vanishes continuously in $\Sigma$ while 
\begin{equation}\label{eqnablaupart}
\lim_{\ve\to0} \nabla u(x+\ve e_n) \to \nabla u(x) \quad \mbox{ in $L^2_{loc}(\sigma|_\Sigma)$}
\end{equation}
with $\nabla u = (\partial_\nu u) \,\nu\in L^2_{loc}(\sigma|_\Sigma)$ defined as a non-tangential 
limit, as shown in Theorem~\ref{teoap1}.
Then, by Green's theorem, using that $u$ is $C^\infty$ in a neighborhood of $\Omega_\ve\cap B(x,r)$ for any $\ve>0$ sufficiently small and $r\in I$, we obtain
\begin{align*}
2\int_{\partial (B(x,r)\cap \Omega_\ve)} u\,\partial_\nu  u\,d\sigma &= \frac1{\sigma(\partial B_r)} \int_{\partial (B(x,r)\cap \Omega_\ve)} \partial_\nu  (u^2)\,d\sigma\\
& 
= \frac1{\sigma(\partial B_r)} \int_{B(x,r)\cap \Omega_\ve} \Delta (u^2)\,dy 
= \frac2{\sigma(\partial B_r)} \int_{B(x,r)\cap \Omega_\ve} |\nabla u|^2\,dy.
\end{align*}
So letting $\ve\to0$, taking into account that $u\in W^{1,2}_{loc}(B)$, 
the third identity in \rf{eq1**} follows.

To show the first identity in \rf{eq1**} observe that,
for all $[a,b]\subset I$, writing $\sigma(\partial B_r)=c_n\,r^{n-1}$,  we have
\begin{align*}
\int_a^b\frac2{\sigma(\partial B_r)} \int_{\partial B(x,r)} u(y)\,\nabla u(y)&\cdot \frac{y-x}r\,d\sigma(y)\,dr\\ & = 2c_n^{-1}\int_a^b \int_{\partial B(x,r)} u(y)\,\nabla u(y)\cdot \frac{y-x}{|y-x|^{n}}\,d\sigma(y)\,dr\\
& = 2c_n^{-1} \int_{A(x,a,b)\cap\Omega} u(y)\,\nabla u(y)\cdot \frac{y-x}{|y-x|^{n}}\,dy\\
& = c_n^{-1} \int_{A(x,a,b)\cap\Omega} {\rm div}_y\left(u(y)^2\frac{y-x}{|y-x|^{n}}\right)\,dy,
\end{align*}
where $A(x,a,b)$ stands for the open annulus centered at $x$ with inner radius $a$ and
outer radius $b$ .
Since $u\in W^{1,2}_{loc}(B)$ and it is smooth in a neighborhood of $B(x,b)\cap \Omega_\ve$, by the divergence theorem, we have
\begin{align*}
c_n^{-1}\int_{A(x,a,b)\cap\Omega} {\rm div}_y&\left(u(y)^2\frac{y-x}{|y-x|^{n}}\right)\,dy 
=c_n^{-1}\lim_{\ve\to0}\int_{A(x,a,b)\cap\Omega_\ve} 
{\rm div}_y\left(u(y)^2\frac{y-x}{|y-x|^{n}}\right)\,dy\\
& \qquad= \lim_{\ve\to0}\left(\frac1{\sigma(\partial B_b)}\int_{\partial B(x,b)\cap\Omega_\ve} u^2\,d\sigma
- \frac1{\sigma(\partial B_a)}\int_{\partial B(x,a)\cap\Omega_\ve} u^2\,d\sigma\right)\\
& \qquad\quad + \lim_{\ve\to0} c_n^{-1} \int_{A(x,a,b)\cap\partial\Omega_\ve} 
u(y)^2\,\frac{\nu(y)\cdot(y-x)}{|y-x|^n}\,d\sigma(y).
\end{align*}
Since $u$ vanishes continuously up to the boundary $\Sigma$, the last limit on the right hand side above vanishes.
Therefore,
\begin{equation}\label{eq*453}
\int_a^b\frac2{\sigma(\partial B_r)} \int_{\partial B(x,r)\cap\Omega} u(y)\,\nabla u(y)\cdot \frac{y-x}r\,d\sigma(y)\,dr = h(x,b)- h(x,a).
\end{equation}
On the other hand, we have already shown that
$$\frac2{\sigma(\partial B_r)} \int_{\partial B(x,r)\cap\Omega} u(y)\,\nabla u(y)\cdot \frac{y-x}r\,d\sigma(y)=\frac2{\sigma(\partial B_r)} \int_{B(x,r)\cap \Omega_\ve} |\nabla u|^2\,dy,$$
and so this term is continuous in $r$, as $u\in W^{1,2}_{loc}(B)$.
Then, the first identity in \rf{eq1**} follows from \rf{eq*453} and the fundamental theorem of calculus.

The identity \rf{eq2**} is an immediate consequence of the definition of $F(x,r)$ and \rf{eq1**}. 
\end{proof}

\vv

The following lemma is already known. It is essentially contained (but not stated in this way) in \cite{AEK}. For the reader's
convenience we include the detailed proof here.

\begin{lemma}\label{lemderiv}
Let $x\in\R^n$ and let $I\subset (0,\infty)$ be a closed bounded interval. Suppose that $B(x,r)\cap\Omega\neq\varnothing$ and $\bar B(x,r)\subset 2B$ for all $r\in I$ (where $B$ is as
Theorem \ref{teomain}). Then
$F(x,\cdot)$ is absolutely continuous in $I$ and, for a.e.\ $r\in I$,
\begin{align}\label{eqcalc}
\partial_rF(x,r) & =  
\frac{4r}{\HH(x,r)^2} \bigg(
\int_{\partial B(x,r)\cap\Omega} |u|^2\,d\sigma\,
 \int_{\partial B(x,r)\cap\Omega} \big|\partial_\nu u\big|^2\,d\sigma - \bigg(\int_{\partial B(x,r)} u\,\partial_\nu u\,d\sigma\bigg)^2\bigg)\\
 &\quad
+  \frac{2}{\HH(x,r)} \int_{B(x,r)\cap\partial\Omega} (y-x)\cdot \nu(y)\,\big|\partial_\nu u(y)\big|^2\,d\sigma(y),\nonumber
\end{align}
where
$$\HH(x,r) = \sigma(\partial B_r) \,h(x,r) = \int_{\partial B(x,r)} u^2\,d\sigma.$$
In particular, if $(y-x)\cdot \nu(y)\geq 0$ for $\sigma$-a.e.\ $y\in B(x,r)\cap\partial\Omega$, 
then $\partial_r F(x,r)\geq0.$
\end{lemma}

\begin{proof}
Denote
$$\II(x,r) = \int_{B(x,r)} |\nabla u|^2\,dy.$$
Since 
$$\II(x,r) = \int_0^r \int_{\partial B(x,r)}  |\nabla u|^2\,d\sigma(y)\,dt,$$
it follows that  $\II(x,\cdot)$ is absolutely continuous with respect to $r$. As $\HH(x,\cdot)$
is of class $C^1$ and bounded away from $0$ in $I$, we deduce that $F(x,\cdot)$
is also absolutely continuous in $I$.

By \rf{eq1**}, we have
$$
\partial_r F(x,r) = \partial_r \frac{2r\, \II(x,r)}{\HH(x,r)} = 2\,\frac{(\II(x,r) + r\,\II'(x,r))\,\HH(x,r) - r\,\II(x,r)\,\HH'(x,r)}{\HH(x,r)^2},
$$
where the symbol $'$ denotes the derivative with respect to $r$.
Observe that, by \rf{eq1**}, for a.e.\ $r\in I$, 
\begin{align*}
\HH'(x,r) & = \partial_r\sigma(\partial B_r)\,h(x,r) + \sigma(\partial B_r)\,h'(x,r) \\
& = \frac{(n-1)\sigma(\partial B_r)}r\,h(x,r)+
 2\,\II(x,r) = \frac{(n-1)}r\,\HH(x,r) + 2\,\II(x,r).
\end{align*}
Therefore,
\begin{equation}\label{eqderf}
\partial_r F(x,r) = \frac2{\HH(x,r)^2} \Big(r\,\HH(x,r)\,\II'(x,r) - (n-2)\HH(x,r)\,\II(x,r)- 2r\,\II(x,r)^2\Big).
\end{equation}

To calculate $\II'(x,r)$ we take into account that
\begin{align*}
\II'(x,r) & = \int_{\partial B(x,r)} |\nabla u|^2\,d\sigma\\
& = \int_{\partial (B(x,r)\cap \Omega)} \frac{y-x}r\cdot\nu(y)\,
|\nabla u(y)|^2\,d\sigma (y)- \int_{B(x,r)\cap \partial\Omega} \frac{y-x}r\cdot\nu(y)|\nabla u(y)|^2\,d\sigma(y).
\end{align*}
 By the Rellich-Necas identity with vector field $\beta(y) = y-x$, $y\in\Omega$, we have
$${\rm div}(\beta\,|\nabla u|^2)= 2\,{\rm div}((\beta\cdot\nabla u)\,\nabla u) + (n-2)\,|\nabla u|^2
\quad \mbox{ in $\Omega.$}
$$
Integrating in $B(x,r)\cap\Omega_\ve$ (with $\Omega_\ve$ as in the proof of Lemma \ref{lem-hxr}),
applying the divergence theorem in this domain, and then letting $\ve\to0$, taking into account \rf{eqnablaupart}, we derive
\begin{align*}
&\int_{\partial(B(x,r)\cap\Omega)} (y-x)\cdot\nu(y)\,|\nabla u(y)|^2\,d\sigma(y)\\
& \qquad=
2\int_{\partial(B(x,r)\cap\Omega)} (y-x)\cdot\nabla u(y)\,\partial_\nu u(y) \,d\sigma(y)+ (n-2)\,\II(x,r)\\
&\qquad = 2r\int_{\partial B(x,r)\cap\Omega} |\partial_\nu u|^2\,d\sigma 
+2\int_{B(x,r)\cap\partial\Omega} (y-x)\cdot\nu(y)\,|\partial_\nu u(y)|^2\,d\sigma(y)
+ (n-2)\,\II(x,r).
\end{align*}
Thus,
\begin{align*}
\II'(x,r) & 
 =2\! \int_{\partial B(x,r)\cap\Omega} \big|\partial_\nu u\big|^2\,d\sigma + \frac{n-2}r\,\II(x,r) + \frac1r \int_{B(x,r)\cap\partial\Omega} (y-x)\cdot \nu(y)\,\big|\partial_\nu u(y)\big|^2d\sigma(y).
\end{align*}

Plugging the last calculation for $\II'(x,r)$ into \rf{eqderf}, we obtain
\begin{align}\label{eqtd1}
\partial_r F(x,r)  =  
\frac{4r}{\HH(x,r)^2} &\bigg(
\HH(x,r) 
 \int_{\partial B(x,r)\cap\Omega} \big|\partial_\nu u\big|^2\,d\sigma \\
& +  \frac{\HH(x,r)}{2r} \int_{B(x,r)\cap\partial\Omega} (y-x)\cdot \nu(y)\,\big|\partial_\nu u(y)\big|^2\,d\sigma(y)
-  \II(x,r)^2\bigg).\nonumber
\end{align}
Observe now that $\II(x,r)$ can be written in the following way:
$$\II(x,r) = \frac12\int_{B(x,r)\cap\Omega} \Delta(u^2)\,dy = \frac12\int_{\partial (B(x,r)\cap\Omega)} \partial_\nu(u^2)\,d\sigma = \int_{\partial B(x,r)} u\,\partial_\nu u\,d\sigma.$$
Plugging the last identity into \rf{eqtd1}, we get \rf{eqcalc}.

The last assertion in the lemma follows from the fact that, by Cauchy-Schwarz,
$$\int_{\partial B(x,r)\cap\Omega} |u|^2\,d\sigma\,
 \int_{\partial B(x,r)\cap\Omega} \big|\partial_\nu u\big|^2\,d\sigma - \bigg(\int_{\partial B(x,r)} u\,\partial_\nu u\,d\sigma\bigg)^2\geq0,$$
and from
the condition that $(y-x)\cdot \nu(y)\geq 0$ for $\sigma$-a.e.\ $y\in B(x,r)\cap\partial\Omega$, which implies that
$$\int_{B(x,r)\cap\partial\Omega} (y-x)\cdot \nu(y)\,\big|\partial_\nu u(y)\big|^2\,d\sigma(y)\geq0.$$
\end{proof}

\vv
It is immediate to check that saying that $\partial_r F(x,r)\geq0$ a.e.\ in an interval is equivalent to saying that the function
$$f(t) = \log h\big(x,e^t\big)$$
is convex in $t=\log r$ for $r$ in that interval, i.e., $f''(\log r)\geq0$ a.e.\ in the interval.
\vv

\begin{lemma}\label{lemconvex}
Given $x\in \R^{n}$,
let $I\subset (0,\infty)$ be an interval such that $h(x,r)>0$ and $\partial_r F(x,r)\geq0$ for a.e.\ $r\in I$.
Given $a>1$, if both $r,ar\in I$, then
\begin{equation}\label{eqconvex}
F(x,r)\leq \log_a \frac{h(x,ar)}{h(x,r)}\leq F(x,ar).
\end{equation}
\end{lemma}

Another way of writing the preceding estimate is the following, for $R=ar$:
\begin{equation}\label{eqconvex2}
h(x,r)\,\left(\frac Rr\right)^{F(x,r)}\leq h(x,R)\leq h(x,r)\,\left(\frac Rr\right)^{F(x,R)}.
\end{equation}
\vv

\begin{proof}
By the convexity of the function $f$ defined above, we have
$$f'(\log r)\leq \frac{f(\log ar) - f(\log r)}{\log ar - \log r} \leq f'(\log ar).$$
It is immediate to check that this is equivalent to \rf{eqconvex}.
\end{proof}
\vv

From now on, we say that an interval  $I\subset (0,\infty)$ is admissible for $x\in\R^{n}$ if $h(x,r)>0$ and $\partial_r F(x,r)\geq0$ for a.e.\ $r\in I$.
\vv

\begin{lemma}\label{lemperturb}
Let $x,y\in\R^{n}$ and $r>0$, $\gamma\in(0,1/10)$, such that $|x-y|\leq\gamma r$. 
Let $I$ be an open interval admissible for $x$ and $y$ such that both $r,\;2(1+\gamma^{1/2})r\in I$.
Suppose that $B(x,5r)\cap\partial\Omega\subset\Sigma$.
Then
\begin{equation}\label{eqpp4}
F(y,r) \leq (1+C\gamma^{1/2}) \,F\big(x, 2(1+\gamma^{1/2})r\big) + C\gamma^{1/2},
\end{equation}
for some absolute constant $C>0$.
\end{lemma}

\begin{proof}
Let $x,y,r,\gamma$ be as above and let $\delta\in (0,1)$ to be chosen below. Since $h(y,\cdot)$ is non-decreasing, we deduce that
$$h(y,r)= \avint_{\partial B(y,r)} u^2\,d\sigma\leq \avint_{A(y,r,(1+\delta) r)} u^2\,dm.$$
Analogously,
$$h(y,r)\geq  \avint_{A(y,(1-\delta) r,r)} u^2\,dm.$$
The same estimates are valid interchanging $y$ by $x$ and/or $r$ by $2r$.
Then, by \rf{eqconvex}, we have
$$F(y,r) \leq \log_2\frac{h(y,2r)}{h(y,r)} \leq 
\log_2\frac{\avint_{A(y,2r,(2+\delta) r)} u^2\,dm}{\avint_{A(y,(1-\delta) r,r)} u^2\,dm}.$$

Observe now that
$$A_y^2 := A(y,2r,(2+\delta) r) \subset A(x,(2-\gamma)r,(2+\delta+\gamma) r) =: A_x^2,$$
and
$$A_y^1 := A(y,(1-\delta) r,r) \supset A(x,(1-\delta+\gamma) r,(1-\gamma)r) =: A_x^1.$$
Thus,
\begin{align*}
F(y,r) & \leq 
\log_2\frac{\avint_{A_y^2} u^2\,dm}{ \avint_{A_y^1} u^2\,dm} \leq
\log_2\left( \frac{\avint_{A_x^2} u^2\,dm}{\avint_{A_x^1} u^2\,dm}\cdot
\frac{m(A_x^2)\,m(A_y^1)}{m(A_x^1)\,m(A_y^2)}\right)\\
& \leq \log_2 \frac{\avint_{\partial B(x,(2+\delta+\gamma) r) } u^2\,dm}{\avint_{\partial B(x,(1-\delta+\gamma) r)} u^2\,dm} + C_{\delta,\gamma},
\end{align*}
where we denoted 
$$C_{\delta,\gamma} = \log_2 \frac{m(A_x^2)\,m(A_y^1)}{m(A_x^1)\,m(A_y^2)}.$$
We choose $\delta=\gamma^{1/2}$. Using just that $\gamma\leq\delta$, we get
$$F(y,r) \leq \log_2 \frac{\avint_{\partial B(x,(2+2\delta) r) } u^2\,dm}{\avint_{\partial B(x,(1-\delta) r)} u^2\,dm} + C_{\delta,\gamma} = \frac{\log \frac{2+2\delta}{1-\delta}}{\log 2}\,
\log_{\frac{2+2\delta}{1-\delta}} \frac{\avint_{\partial B(x,(2+2\delta) r) } u^2\,dm}{\avint_{\partial B(x,(1-\delta) r)} u^2\,dm} + C_{\delta,\gamma}
.$$
Since $(2+2\delta)r\in I$ (by assumption), by Lemma \ref{lemconvex} we have
$$\log_{\frac{2+2\delta}{1-\delta}} \frac{\avint_{\partial B(x,(2+2\delta) r) } u^2\,dm}{\avint_{\partial B(x,(1-\delta) r)} u^2\,dm}\leq F(x,(2+2\delta) r).$$
It is also immediate to check that
$$\frac{\log \frac{2+2\delta}{1-\delta}}{\log 2} \leq 1 + C\delta = 1+ C\gamma^{1/2}.$$
Hence,
$$F(y,r) \leq (1+ C\gamma^{1/2}) F(x,(2+2\gamma^{1/2}) r) + C_{\delta,\gamma}.$$

It only remains to show that $C_{\delta,\gamma}\leq  C\gamma^{1/2}$. To this end, observe that
\begin{align*}
\frac{m(A_x^2)}{m(A_y^2)} &= \frac{(2+\delta+\gamma)^n - (2-\gamma)^n}{(2+\delta)^n- 2^n}\\
& = \frac{(2+\delta)^n + n(2+\delta)^{n-1}\gamma + O(\gamma^2) - 2^n +n2^{n-1}\gamma + O(\gamma^2)}{(2+\delta)^n- 2^n}\\
& = 1 + 
 \frac{ n(2+\delta)^{n-1}\gamma  + n2^{n-1}\gamma + O(\gamma^2)}{n2^{n-1}\delta + O(\delta^2)}.
\end{align*}
It follows that
$$\bigg|\frac{m(A_x^2)}{m(A_y^2)} - 1\bigg|\leq C\,\frac{\gamma}\delta = C\gamma^{1/2}.$$
Almost the same arguments show that
$$\bigg|\frac{m(A_y^1)}{m(A_x^1)} - 1\bigg|\leq C\,\frac{\gamma}\delta = C\gamma^{1/2}.$$
Therefore,
$$C_{\delta,\gamma} = \log_2\frac{m(A_x^2)}{m(A_y^2)} +  \log_2\frac{m(A_y^1)}{m(A_x^1)}\lesssim\gamma^{1/2},$$
as wished.
\end{proof}

\vv


\section{The Key Lemma}\label{seckey}

To prove Theorem \ref{teomain} we consider an arbitrary ball $B_0$ centered in $\Sigma$ such that 
$M^2B_0\subset B$, where
$B$ is the ball in Theorem \ref{teomain} and $M\gg1$ will be fixed below. We denote 
 $\Sigma_0 =
\partial \Omega \cap MB_0$. We will show that if $u$ is a non-zero (i.e., not identically zero) function harmonic in $\Omega$ and continuous in $\overline\Omega$ which vanishes in $\Sigma$, then the normal derivative $\partial_\nu u$ 
cannot vanish
in a subset of $\Sigma_0\cap B_0$ with positive surface measure.
Clearly, this suffices to prove Theorem \ref{teomain}.

Let $H_0$ be the horizontal hyperplane through the origin.
By the hypotheses in the theorem, we can assume that $\partial \Omega\cap MB_0$ is a Lipschitz graph with respect to the hyperplane $H_0$ with slope at most $\tau_0\ll1$, and that $\Omega\cap MB_0$
is above the graph.
We 
consider the following Whitney decomposition of $\Omega$: we have a family $\WW$ of dyadic cubes in $\R^n$ with disjoint interiors such that
$$\bigcup_{Q\in\WW} Q = \Omega,$$
and moreover there are
 some constants $\Lambda>20$ and $D_0\geq1$ such the following holds for every $Q\in\WW$:
\begin{itemize}
\item[(i)] $10Q \subset \Omega$;
\item[(ii)] $\Lambda Q \cap \partial\Omega \neq \varnothing$;
\item[(iii)] there are at most $D_0$ cubes $Q'\in\WW$
such that $10Q \cap 10Q' \neq \varnothing$. Further, for such cubes $Q'$, we have $\frac12\ell(Q')\leq \ell(Q)\leq 2\ell(Q')$.
\end{itemize}
Above, we denote by $\ell(Q)$ the side length of $Q$. 
From the properties (i) and (ii) it is clear that $\dist(Q,\partial\Omega)\approx\ell(Q)$. We assume that
the Whitney cubes are small enough so that
\begin{equation}\label{eqeq29}
\diam(Q)< \frac1{20}\,\dist(Q,\partial\Omega).
\end{equation}
The arguments to construct a Whitney decomposition satisfying the properties above are
standard but we include the detailed arguments in Lemma 
\ref{lem-whitney} below for the convenience of the reader.


Let $\Pi$ denote the
orthogonal projection on $H_0$.
By translating the usual dyadic
lattice if necessary, we can
assume that there exists some cube $R_0\in\WW$ such that $\Pi(B_0)\subset \Pi(R_0)$ and  $\ell(R_0)\leq C\,r(B_0)$ and moreover $R_0 \subset \frac M2B_0$, for $M$ big enough.

Next we need to define some ``generations" of cubes in $\WW$.
We let $\DD_\WW^0(R_0) = \{R_0\}$. For $k\geq 1$ we define $\DD_\WW^k(R_0)$ as follows.
Let 
\begin{equation}\label{eqjr0}
J(R_0) = \{\Pi(Q):Q\in\WW \mbox{ such that $\Pi(Q)\subset \Pi(R_0)$ and $Q$ is below $R_0$}\}.
\end{equation}
Observe that $J(R_0)$ is a family of $(n-1)$-dimensional dyadic cubes in $H_0$, all of them contained in 
$\Pi(R_0)$. Let  $J_k(R_0)\subset J(R_0)$ be the subfamily of $(n-1)$-dimensional dyadic cubes in $H_0$ with side length equal to $2^{-k}\ell(R_0)$.
 To each $Q'\in J_k(R_0)$ we assign some $Q\in\WW$ such that $\Pi(Q)=Q'$, $\Pi(Q)\subset \Pi(R_0)$, and such that $Q$ is below $R_0$ (see Lemma \ref{lemwhit2} for more details), and we write $s(Q')=Q$. Notice there may be
more than one possible choice for $Q$. However, the choice is irrelevant. Anyway, for definiteness
we take the cube $Q$ that is closest to $R_0$ among all the possible choices. Then we define
$$\DD_\WW^k(R_0) = \{s(Q'):Q'\in J_k(R_0)\}.$$
 Next we let 
$$\DD_\WW(R_0) = \bigcup_{k\geq0} \DD_\WW^k(R_0)\}.$$
Notice that, for each $k$, the family $\{\Pi(Q):Q\in\DD_\WW^k(R_0)\}$ is a partition of $\Pi(R_0)$.
Finally, for each $R\in\DD_\WW^k(R_0)$ and $j\geq1$ we denote
\begin{equation*}
\DD_\WW^j(R) = \{Q\in\DD_\WW^{k+j}(R_0):\Pi(Q)\subset\Pi(R)\}.
\end{equation*}


By the properties of the Whitney cubes, it is easy to check that
$$Q\in\DD_\WW(R_0)\quad \Rightarrow\quad \dist(Q,\Sigma_0) = \dist(Q,\partial\Omega)
\approx\ell(Q).$$
From now on, for any cube $Q$, we denote by $x_Q$ its center. Further, we denote
by $m_{n-1}$ the $(n-1)$-dimensional Lebesgue measure on the hyperplane $H_0$.

\begin{keylemma}\label{keylemma}
Under the assumptions of Theorem \ref{teomain},
let $R_0$ be as above and let $N_0>1$ be big enough. There exists some absolute constant $\delta_0>0$ such that for all $A\gg1$ big enough the following holds, assuming also $\tau_0$ small enough
and $M$ big enough.
Let $R\in \DD_\WW(R_0)$ satisfy $F(x_R,A\,\ell(R))\geq N_0$.
There exists some positive integer $K=K(A)$ big enough such that if we let 
$$\GZ_K(R)=\big\{Q\in\DD_\WW^K(R):F(x_Q,A\ell(Q))\leq \tfrac12 \, F(x_R,A\,\ell(R))\big\},$$
then:
\begin{itemize}
\item[(a)] $m_{n-1}\Big(\bigcup_{Q\in \GZ_K(R)} \Pi(Q)\Big) \geq \delta_0\,m_{n-1}(\Pi(R))$.

\item[(b)] For all $Q\in \DD_\WW^K(R)$, it holds
$$F(x_Q,A\ell(Q))\leq (1+CA^{-1/2})\,F(x_R,A\ell(R)).$$
\end{itemize}
\end{keylemma}

A key point in the lemma is that $\delta_0$ does not depend on $A$. On the other hand, $\tau_0$, $M$, and $K$
depend on $A$. The constant $N_0$ is also an absolute constant independent of the other parameters.

The general strategy for the proof of the Key Lemma is similar to one of the Hyperplane Lemma 4.1 from \cite{Logunov1}. The main differences stem from the fact that in the lemma above we wish to estimate the frequency function
in points that are close to $\partial\Omega$, and then we have to be more careful and more precise with the monotonicity
properties of the frequency function.

A basic tool for the proof of the Key Lemma \ref{keylemma} is the following result on quantitative Cauchy uniqueness:

\begin{theorem}\label{teoARRV}
Let $v$ be a function harmonic in the half ball $B_1^+ =\{x\in\R^n:|x|<1,x_n>0\}$ and $C^1$ smooth up to
the boundary. Let $\Gamma$ the following subset of $\partial B_1^+$:
$$\Gamma =\{x\in\R^n:|x|<3/4,x_n=0\}.$$
Suppose that
$$\int_{B_1^+}|v|^2\,dm\leq 1$$
and
$$\sup_{\Gamma}|v| + \sup_{\Gamma}|\nabla v| \leq \ve,$$
for some $\ve\in(0,1)$. Then
$$\sup_{B(1/2,1/4)}|v|\leq C\ve^\alpha,$$
where  $C,\alpha$ are positive absolute constants.
\end{theorem}

This result appears in \cite[Lemma 4.3]{Lin} and it is proven in much greater generality in 
\cite[Theorem 1.7]{ARRV}.
\vv

\begin{rem}\label{rem1}
We claim $T>0$ and $R\in\DD_\WW(R_0)$, if $x\in\Omega$ satisfies
$$\dist(x,R)\leq  T\,\ell(R) \quad\text{ and }\quad \dist(x,\partial \Omega)\geq T^{-1}\,\ell(R),$$
then the interval $(0,A\ell(R))$ is admissible for $x$, assuming $M\gg T,A$ and that $\tau_0$ is small enough, depending on $T$ and $A$. This property will be essential for the proof of the Key Lemma.

To prove the claim it suffices to check that, in the situation above,
\begin{equation}\label{eqrem*20}
(y-x)\cdot \nu(y)\geq 0 \quad \mbox{ $\sigma$-a.e.\ $y\in B(x,r)\cap\partial\Omega$, $0<r\leq A\ell(R)$,
}
\end{equation}
since then Lemma \ref{lemderiv} ensures that $\partial_r F(x,r)\geq0$ for $0<r\leq A\ell(R)$.
To prove \rf{eqrem*20}, let $x'\in\Sigma$ be the point such that $\Pi(x)=\Pi(x')$, so that $x-x'$ is orthogonal to $H_0$. Denote by $H_{x'}$ the hyperplane parallel to $H_0$ through $x'$. Given
$y$ as in \rf{eqrem*20}, let $y'$ be the orthogonal projection of $y$ on $H_{x'}$. See fig.\ \ref{fig*}.
So $y'-x'$ is orthogonal to $x-x'$. Moreover, since the slope of $\Sigma$ is at most $\tau_0$,
$$|y-y'|\leq \tau_0\,|y'-x'|\leq \tau_0\,|y-x|\quad \mbox{ and }\quad|\nu(y) - (-e_n)| \leq \tau_0.$$
Then we write
\begin{align*}
(y-x)\cdot \nu(y) & = \big[(y - y') + (y'-x') + (x'-x)\big] \cdot  \big[-e_n + (e_n + \nu(y))\big]\\
& \geq |x-x'| - 
|x'-x| \, |e_n+\nu(y)| - |y-y'| - |y'-x'| \,|e_n + \nu(y)| \\
& \geq  |x'-x|- \tau_0\,|x-x'|- \tau_0\,|y-x| - \tau_0\,|y-x| \\
&= (1-\tau_0)\,|x'-x| - 2\tau_0\,|y-x|.
\end{align*}
Using now that, for $\tau_0\leq 1/2$, we have $|x'-x|\approx\dist(x,\partial\Omega)$, we get
$$(y-x)\cdot \nu(y)\geq c\,\dist(x,\partial\Omega) -2\tau_0\,|y-x|
\geq c\,T^{-1}\,\ell(R) - 2 \,\tau_0\,A\,\ell(R)\geq0,$$
assuming $\tau_0\leq \frac c2\,T^{-1}$, which proves \rf{eqrem*20}.
\end{rem}

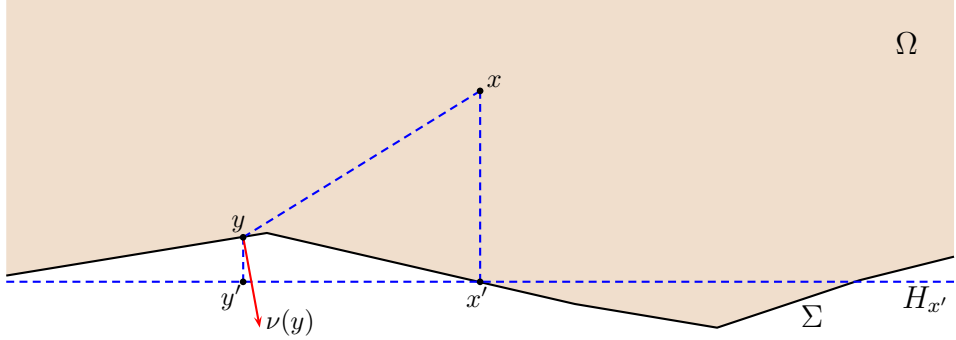
\begin{figure}
\psset{unit=0.63cm}
\newgray{superlightgray}{0.9}

\begin{center}
\begin{pspicture}(20,8)

\pscustom{  
 \psline(0,1.1)(5.5,2)(12,0.5)(15,0)(18,1)(20,1.5) 
 
\gsave
 \psline[linecolor=superlightgray,liftpen=2](20,1.5)(20,7)(0,7)(0,1.1)
\fill[fillstyle=solid,fillcolor=almond]
\grestore 
}

\psline[linecolor=blue,linestyle=dashed,dash=3pt 2pt](10,5)(5,1.91)
\psline[linecolor=blue,linestyle=dashed,dash=3pt 2pt](5,1.91)(5,0.97)

\psline[linecolor=blue,linestyle=dashed,dash=3pt 2pt](10,5)(10,0.97)
\psline[linecolor=blue,linestyle=dashed,dash=3pt 2pt](0,0.97)(20,0.97)
\psline[linecolor=red]{->}(5,1.91)(5.35,0)

\psdots*[dotscale=0.7](10,5)
\psdots*[dotscale=0.7](10,0.97)
\psdots*[dotscale=0.7](5,1.91)
\psdots*[dotscale=0.7](5,0.97)
\rput{0}(19,6){\large $\Omega$}
\rput{0}(17,0.25){\large $\Sigma$}
\rput{0}(19.4,0.6){\large $H_{x'}$}
\rput{0}(10.3,5.2){\small $x$}
\rput{0}(4.9,2.2){\small $y$}
\rput{0}(4.75,0.64){\small $y'$}
\rput{0}(9.94,0.65){\small $x'$}
\rput{0}(6,0.1){{\small{$\nu(y)$}}}
 \end{pspicture}
\vspace{5mm}

 \caption{ \label{fig*} The domain $\Omega$ and the Lipschitz graph $\Sigma$. The points $x,y$ satisfy
 $(y-x)\cdot \nu(y)\geq 0$.}
 
\end{center}

\end{figure}

\vv

\begin{proof}[\bf Proof of the Key Lemma \ref{keylemma}]
For any $Q\in\WW$, we consider its associated cylinder:
$$\CC(Q) = \Pi^{-1}(\Pi(Q)).$$

Let $R\in \DD_\WW(R_0)$ be as in the lemma and let $A\gg1$.
For some $j\gg1$ to be fixed below (independent of $A$), 
let $L$ be a hyperplane parallel to $H_0$ such that
$$\dist(L,\Sigma_0\cap \CC(R)) = 2^{-j}\,\ell(R).$$
Notice that there are two possible choices for $L$.
If $\tau_0$ is small enough (and so $\Sigma$ flat enough) depending on $j$, then
\begin{equation}\label{eqdist1}
\dist(x,\partial\Omega\cap \CC(10R))\approx 2^{-j}\,\ell(R)\quad\mbox{ for all $x\in L\cap \CC(10R)$.}
\end{equation}
Then we choose $L$ so that $L\cap \CC(10R)\subset \Omega$.

Let $J$ denote the family of cubes from $\WW$ which intersect $L\cap \CC(\frac12R)$. 
By the properties of Whitney cubes and \rf{eqdist1}, it is clear that
$$\ell(Q)\approx 2^{-j}\,\ell(R)\quad \text{and} \quad \Pi(Q)\subset\Pi(R)
\quad\mbox{ for all $Q\in J$.}$$
Denote by ${\rm Adm}(2\Lambda Q)$ the set of points  
$x\in \Omega\cap 2\Lambda Q$ such that the interval $(0,\diam(25\Lambda Q))$ is admissible for $x$.
Recall that $\Lambda$ is one of the constants in the definition of Whitney cubes.
We assume $\tau_0$ small enough so that $3Q\subset {\rm Adm}(2\Lambda Q)$\footnote{Notice
that $2\Lambda Q\not\subset{\rm Adm}(2\Lambda Q)$ because $2\Lambda Q$ intersects $\R^n\setminus
\overline\Omega$. Instead, a big portion of $2\Lambda Q$ is contained in ${\rm Adm}(2\Lambda Q)$ if $\Sigma_0$ is flat enough.}.
Then by Lemma \ref{lemperturb}, 
\begin{equation}\label{eqpert1}
\sup_{x\in {\rm Adm}(2\Lambda Q)} F(x,\diam(5\Lambda Q))\leq C_0\,F(x_Q,\diam(20\Lambda Q)) + C_0,
\end{equation}
where $C_0$ is an absolute constant.

\begin{claim*}
There exists some $Q\in J$ such that 
\begin{equation}\label{eqclaim49}
F(x_Q,\diam(20\Lambda Q))\leq \frac{F(x_R,A\ell(R))}{4C_0}
\end{equation}
if $j$ is big enough and
we assume that $\tau_0$ is small enough depending on $j$, and also that $N_0$ is big enough.
\end{claim*}

Remark again that the choice of $j$ will not depend on the constant $A$.

To prove the claim we intend to apply a rescaled version of Theorem \ref{teoARRV} to a suitable half ball $B_+$ centered
at $z_R$, the orthogonal projection of $x_R$ on $L$. We take
$$B_+ = \big\{x\in B(z_R,\ell(R)/4):x_n> (z_R)_n\big\},$$
so that $B_+\subset \Omega$, assuming that $\Sigma_0\cap\CC(R)$ is below $L$. We also consider the point
$$\wt z_R = z_R + (0,\ldots,0,\ell(R)/8).$$
Notice that $\wt z_R\in B_+$ (in fact, $B(\wt z_R,\ell(R)/8)\subset B_+$).

Aiming for a contradiction, suppose that $F(x_Q,\diam(20\Lambda Q))> \frac N{4C_0}$ for all $Q\in J$, where 
$N=F(x_R,A\ell(R))$.
For each $Q\in J$, by the subharmonicity of $|u|$ and \rf{eqconvex2}, we have
\begin{align*}
\sup_{2Q}|u|  &\lesssim \avint_{\partial B(x_Q,\diam(3Q))}\!\!|u|\,d\sigma \leq h(x_Q,\diam(20\Lambda Q))^{1/2}\\
&\leq h(x_Q,\ell(R))^{1/2}
\left(\frac{\diam(20\Lambda Q)}{\ell(R)}\right)^{F(x_Q,\diam(20\Lambda Q))/2}.
\end{align*}
Here we applied the property described in Remark \ref{rem1}, allowing the smallness of the slope constant $\tau_0$ to depend on $j$. Below we will make repeated use of this property, often without further reference.

To estimate $h(x_Q,\ell(R))$ we take into account that
$$h(x_Q,\ell(R))\leq \avint_{A(x_Q,\ell(R),2\ell(R))}|u|^2\,dm \lesssim \avint_{B(\wt z_R,C_1\ell(R))}|u|^2\,dm
\leq h(\wt z_R,C_1\ell(R)),
$$
since $A(x_Q,\ell(R),2\ell(R))\subset B(\wt z_R,C_1\ell(R))$ for some fixed $C_1>1$.
Further, by \rf{eqconvex2},
\begin{align*}
h(\wt z_R,C_1\ell(R)) & \leq h(\wt z_R,\ell(R)/16) \,(16\, C_1)^{F(\wt z_R,C_1\ell(R))}\\
& \lesssim 
(16\, C_1)^{F(\wt z_R,C_1\ell(R))}\,\avint_{B(\wt z_R,\ell(R)/8)} |u|^2\,dm\\
&\lesssim 
(16\, C_1)^{F(\wt z_R,C_1\ell(R))}\,\avint_{B_+} |u|^2\,dm,
\end{align*}
recalling that $B(\wt z_R,\ell(R)/8)\subset B_+$.
Thus,
\begin{equation}\label{eqkey47}
\sup_{2Q}|u|^2 \lesssim (16\, C_1)^{F(\wt z_R,C_1\ell(R))}\,\left(\frac{\diam(20\Lambda Q)}{\ell(R)}\right)^{F(x_Q,\diam(20\Lambda Q))}\avint_{B_+} |u|^2\,dm.
\end{equation}

Observe now that, by Lemma \ref{lemperturb}, 
$$F(\wt z_R,C_1\ell(R)) \leq C\,F(x_R,C\ell(R)) + C,$$
for a suitable absolute constant $C>2C_1$. So for $A$ and $N_0$ big enough (both independent of $j$, just larger than some absolute constant),
\begin{equation}\label{eqff2}
F(\wt z_R,C_1\ell(R)) \leq C\,F(x_R,A\ell(R)) + C\leq C'\,N.
\end{equation}
Therefore, recalling also the assumption $F(x_Q,\diam(20\Lambda Q))> \frac N{4C_0}$, by \rf{eqkey47} we get
\begin{align*}
\sup_{2Q}|u|^2 & \lesssim (16\, C_1)^{C'N}\,\left(\frac{\diam(20\Lambda Q)}{\ell(R)}\right)^{N/4C_0}
\avint_{B_+} |u|^2\,dm  = 2^{-jcN+C''N}\avint_{B_+} |u|^2\,dm
\end{align*}
(here we took into account that $\diam(20\Lambda Q)\leq\ell(R)$ for $j$ larger than some absolute constant).
Also, by standard interior estimates for harmonic functions,
$$\sup_{\frac32Q}|\nabla u|^2 \lesssim \frac1{\ell(Q)^2}\sup_{2Q}|u|^2 
\lesssim \frac{2^{2j}}{\ell(R)^2}\,2^{-jcN+C''N}\avint_{B_+} |u|^2\,dm.$$
From the last two estimates we deduce that if $j$ is big enough and $N_0$ (and thus $N$) also big enough, then there exists some $c'>0$
such that
$$\sup_{\frac32 Q}\big(|u|^2 + \ell(R)^2\,|\nabla u|^2\big) \lesssim 2^{-jc'N}\avint_{B_+} |u|^2\,dm.$$
Since the cubes $\frac32 Q$ with $Q\in J$ cover the flat part of the boundary of $B_+$, which we denote by $\Gamma$, it is clear that
$$\sup_{\Gamma}\big(|u|^2 + \ell(R)^2\,|\nabla u|^2\big) \lesssim 2^{-jc'N}\avint_{B_+} |u|^2\,dm.$$

Applying now a rescaled version of Theorem \ref{teoARRV} to the half ball $B_+$, we infer that
$$\sup_{B(\wt z_P,\ell(R)/16)} |u|^2 \lesssim 2^{-jc'N\alpha} \avint_{B_+} |u|^2\,dm.$$
Consequently,
$$h(\wt z_R,\ell(R)/16)\lesssim  2^{-2jc'N\alpha} \avint_{B(\wt z_R,\ell(R))} |u|^2\,dm
\lesssim 2^{-2jc'N\alpha} \,h(\wt z_R,\ell(R)),$$
for some fixed $\alpha>0$.
By Lemma \ref{lemconvex}, this implies that
$$F(\wt z_R,\ell(R))\geq \log_{16}\frac{h(\wt z_R,\ell(R))}{h(\wt z_R,\ell(R)/16)}\geq c\,jN\alpha,$$
for some fixed $c>0$.
However, for $j$ big enough this contradicts the fact that $F(\wt z_R,\ell(R))\lesssim N$, which follows from \rf{eqff2}. So the proof of the claim is concluded.

\vv
Now we are ready to introduce the set  $\GZ_K(R)$.
Fix $Q_0\in J$ such that \rf{eqclaim49} holds for $Q_0$. 
Notice that, by \rf{eqpert1}, 
\begin{equation}\label{eqadm45}
\sup_{x\in {\rm Adm}(2\Lambda Q_0)} F(x,\Lambda\ell(Q_0))\leq C_0\,F(x_{Q_0},\diam(20\Lambda Q_0)) + C_0
\leq \frac N4 + C_0\leq \frac N2,
\end{equation}
since $N\ge N_0$ and we assume $N_0$ big enough.
Now we just define
$$\GZ_K(R) = \{Q\in\DD_\WW^{j+k}(R):\Pi(Q)\subset \Pi(Q_0)\},$$
with $k=\lceil\log_2 A\rceil$. So we have $\GZ_K(R)\subset \DD_\WW^K(R)$ with $K=j+k$ and it holds
$\ell(Q)\approx 2^{-k}\ell(Q_0)$ for every $Q\in \GZ_K(R)$.

The property (a) in the lemma follows easily from \rf{eqadm45}. Indeed, if $P\in \GZ_K(R)$,
then taking into account that $x_P\in{\rm Adm}(2\Lambda Q_0)$ for $\tau_0$ small enough (depending on $A$),
$$F(x_P,A\ell(P))\leq F(x_P,\ell(Q_0))\leq F(x_P,\Lambda\ell(Q_0)) \leq \frac N2.$$
Notice also that 
$$m_{n-1}\Big(\bigcup_{Q\in \GZ_K(R)} \Pi(Q)\Big) = \ell(Q_0)^{n-1} \approx (2^{-j}\,\ell(R))^{n-1},
$$
and recall that $j$ is independent of $A$. So (a) holds with $\delta_0\approx 2^{-j(n-1)}$.

The property (b) is an easy consequence of Lemma \ref{lemperturb}.
Indeed, for any $P\in\DD_\WW^K(R)$, since 
$|x_P- x_R|\lesssim \ell(R)$, taking $\gamma\approx A^{-1}$ in \rf{eqpp4}, we deduce
\begin{align*}
F(x_P,A\ell(P))&\leq F(x_P,A\ell(R)/3) \leq (1+C A^{-1/2}) \,F\big(x_R, A\ell(R)\big) + CA^{-1/2}\\
&\leq (1+2C A^{-1/2}) \,N.
\end{align*}
\end{proof}
\vv


\section{The proof of Theorem \ref{teomain}}

Our next objective is to prove the following result:

\begin{lemma}\label{lemliminf}
Under the assumptions of Theorem \ref{teomain}, let $R_0\in \WW$ be as in Section \ref{seckey}.
Then, 
$$\liminf_{r\to 0} \frac{h(x,12r)}{h(x,r)} <\infty
\quad\mbox{ for $\sigma$-a.e.\ $x\in \Sigma_0 \cap \CC(R_0)$.}$$
\end{lemma}

Recall that $\CC(R_0)$ is the cylinder
$$\CC(R_0) = \Pi^{-1}(\Pi(R_0)),$$
and it contains $B_0$, by the assumption just after \rf{eqeq29}.
\vv

The proof of the preceding lemma will use the following version of the law of large numbers, due to
Etemadi \cite{Etemadi}:

\begin{theorem}\label{teoete}
Let $\{ X_k\}_{k\geq1}$ be a sequence of non-negative random variables with finite second
moments such that:
\begin{itemize}
\item[(a)] $\sup_{k\geq1} \E X_k< \infty$,
\item[(b)] $\E(X_j\,X_k) \leq \E X_j\,\E X_k$ for $j\neq k$, and
\item[(c)] $\sum_{k\geq1} \frac1{k^2}\,{\rm Var} X_k <\infty$.
\end{itemize}
Let $S_m = X_1+\ldots+X_m$. Then
$$\lim_{m\to\infty} \frac{S_m - \E S_m}{m} = 0 \quad \mbox{ almost surely}.$$
\end{theorem}

\vv
\begin{proof}[\bf Proof of Lemma \ref{lemliminf}]
Let $\Pi_{\Sigma_0}:\CC(R_0)\to \CC(R_0)\cap\Sigma_0$ denote the projection on $\CC(R_0)\cap\Sigma_0$ in the direction orthogonal to the
horizontal hyperplane $H_0$. We consider the measure $$\mu=\Pi_{\Sigma_0}\# (m_{n-1}|_{\CC(R_0)\cap H_0}).$$ This is the image measure (or push-forward measure) of the $(n-1)$-dimensional Lebesgue measure on $\CC(R_0)\cap H_0$ to $\CC(R_0)\cap\Sigma_0$. Obviously, $\mu$ is mutually absolutely continuous with 
the surface measure $\sigma$ on $\CC(R_0)\cap\Sigma_0$. 

Next we consider the families of cubes from $\WW$ defined by
\begin{equation}\label{deftj}
\TT_j' = \big\{ Q\in\DD_\WW^{jK}(R_0): F(x_Q,A\ell(Q))\leq N_0\big\}, \qquad j\geq0,
\end{equation}
where the constants $K$, $A$, and $N_0$ are given by the Key Lemma \ref{keylemma} (the precise large value of $A$
will be chosen below).
We also denote
$$R_{\Sigma_0} = \Pi_{\Sigma_0}(R_0)$$
and consider the following subset of $R_{\Sigma_0}$:
$$T_j = \bigcup_{Q\in\TT_j'} \Pi_{\Sigma_0}(Q),\qquad j\geq 0.$$
We will prove the following:

\begin{claim*}
We have
$$\mu\Big(R_{\Sigma_0} \setminus \limsup_{j\to\infty}  T_j\Big)=0.$$
\end{claim*}

Let us see first that the lemma follows from this claim. Indeed, if $x\in T_j$, then there exists
some cube $Q\in\TT_j'$ such that $x\in \Pi_{\Sigma_0}(Q)$.
By construction, $F(x_Q,A\ell(Q))\leq N_0$ and thus, by \rf{eqconvex2},
$$h(x_Q,A\ell(Q))\leq 48^{N_0}\,h(x_Q,A\ell(Q)/48).$$
For $A$ big enough, we have 
$$B(x,A\ell(Q)/24)\supset \tfrac32 B(x_Q,A\ell(Q)/48) \quad\mbox{ and } \quad
\tfrac32 B(x,A\ell(Q)/2)\subset B(x_Q,A\ell(Q)).$$
Then, by the subharmonicity of $|u|$ in a neighborhood of $\Sigma_0$ and standard arguments,
$$h(x,A\ell(Q)/24)\gtrsim h(x_Q,A\ell(Q)/48)\quad \mbox{ and } \quad
h(x,A\ell(Q)/2)\lesssim h(x_Q,A\ell(Q)).$$
Hence, for each $x\in T_j$,
$$\frac{h(x,A\ell(Q)/2)}{h(x,A\ell(Q)/24)} \lesssim 
\frac{h(x_Q,A\ell(Q))}{h(x_Q,A\ell(Q)/48)} \leq 48^{N_0},$$
with $\ell(Q)\approx 2^{-jK}\ell(R_0)$.

Consequently, if $x\in\limsup_{j\to\infty}  T_j$, then there exists a sequence of radii $r_j\to0$ such that
$$\frac{h(x,12r_j)}{h(x,r_j)} \lesssim  48^{N_0},$$
which implies that
$$\liminf_{r\to 0} \frac{h(x,12r)}{h(x,r)} <\infty,$$
and yields the lemma, assuming the claim.

\vv
To prove the claim above we need to introduce some additional notation. For $j\geq 0$ and $K$ as in
the Key Lemma \ref{keylemma}, we denote
$$\wt\DD_j(R_{\Sigma_0})= \Pi_{\Sigma_0} (\DD_\WW^{jK}(R_0)),$$
or more precisely,
$$\wt \DD_j(R_{\Sigma_0}) = \big\{\Pi_{\Sigma_0}(Q'):Q'\in \DD_\WW^{jK}(R_0)\big\}.$$
We also set
$$\wt \DD(R_{\Sigma_0}) = \bigcup_{j\geq0}\wt \DD_j(R_{\Sigma_0}).$$
For any $R\in\wt\DD_j(R_{\Sigma_0})$ such that $R=\Pi_{\Sigma_0}(R')$ for some $R'\in
 \DD_\WW^{jK}(R_0)$, in case that $F(x_{R'},A\,\ell(R'))\geq N_0$ 
 we consider the good set
$$G(R) = \bigcup_{Q'\in \GZ_K(R')} \Pi_{\Sigma_0}(Q'),$$
with $\GZ_K(R')$ as in the Key Lemma \ref{keylemma}. In case that $F(x_{R'},A\,\ell(R'))< N_0$
we let 
$$G(R) =  \Pi_{\Sigma_0}(R')=R.$$
 Finally, we write
$$\TT_j = \big\{ \Pi_{\Sigma_0}(Q'):Q'\in\DD_\WW^{jK}(R_0), F(x_{Q'},A\ell(Q'))\leq N_0\big\},$$
or in other words, 
$$\TT_j=\Pi_{\Sigma_0}(\TT_j'),$$
with $\TT_j'$ defined in \rf{deftj}

To prove the claim we have to show that
$\bigcup_{j\geq h} T_j$ has full $\mu$-measure in $R_{\Sigma_0}$ for every $h\geq0$. To this
end, for any fixed $h$ we define the following functions $f_j$, $j\geq h$:
$$f_j = \sum_{Q\in\wt D_j(R_{\Sigma_0})} f_Q,$$
where $f_Q=0$ if $Q$ is contained in some ``cube'' $\wt Q\in\bigcup_{j\geq h}\TT_j$ and, otherwise,
$$f_Q= \frac{\mu(Q)}{\mu(G(Q))}\,\chi_{G(Q)}- \chi_Q.$$
It is immediate to check that the functions $f_j$ have zero $\mu$-mean and they are orthogonal, i.e.,
$\int f_i\,f_j\,d\mu=0$ if $i\neq j$ (taking into account that $f_j$ has zero $\mu$-mean in each
$Q\in\wt\DD_{j}(R_{\Sigma_0})$ and
is constant in each $P\in\wt\DD_{j+1}(R_{\Sigma_0})$). Observe also that the functions $f_j$ are uniformly bounded, due
to the fact that $\mu(G(Q))\geq \delta_0\,\mu(Q)$ in the latter case by the Key Lemma. So their $L^2(\mu)$
norms are uniformly bounded too.

We consider the probability measure $\mu_{|R_{\Sigma_0}}/\mu(R_{\Sigma_0})$ and
the random variables $X_j= f_j+1$, $j\geq h$. Notice that they are non-negative and the assumptions
in Theorem \ref{teoete} are satisfied. Indeed, (a) and (c) follow from the uniform boundedness of the
functions $f_j$, and the zero mean of each $f_j$ and the mutual orthogonality of the $f_j$'s imply that
$\E(X_i\,X_j) = \E(X_i)\,\E(X_j)$ if $i\neq j$. Applying the theorem then we infer that
\begin{equation}\label{eqlaw}
\lim_{m\to\infty} \frac1m \sum_{j= h+1}^m f_j(x) = 0\quad \mbox{ for $\mu$-a.e.\ $x\in R_{\Sigma_0}$},
\end{equation}
using the fact that $\E X_j=1$ for all $j$.

We will show that 
\begin{equation}\label{eqimp5}
x\in R_{\Sigma_0}\setminus \bigcup_{j\geq h} T_j\qquad \Rightarrow\qquad\lim_{m\to\infty} \frac1m \sum_{j= h+1}^m f_j(x) \neq 0.
\end{equation}
Clearly, by \rf{eqlaw}, this implies that 
$R_{\Sigma_0}\setminus \bigcup_{j\geq h} T_j$ has null $\mu$-measure and finishes the proof of the claim.

We prove \rf{eqimp5} by contradiction. Suppose that there exists some point
$x\in R_{\Sigma_0}\setminus \bigcup_{j\geq h} T_j$ such that $\lim_{m\to\infty} \frac1m \sum_{j= h+1}^m f_j(x) = 0$.
Denote by $Q_j$ the ``cube'' from $\wt D_j(R_{\Sigma_0})$ that contains $x$. Since $Q_i\not\in \TT_i$
for any $i\geq h$, by definition we have
$$f_j(x) = \frac{\mu(Q_j)}{\mu(G(Q_j))}\,\chi_{G(Q_j)}(x) - 1\quad\mbox{ for any $j\geq h$.}$$
Then \rf{eqlaw} tells us that, for any $\ve>0$, 
$$\bigg| \sum_{j= h+1}^m  \frac{\mu(Q_j)}{\mu(G(Q_j))}\,\chi_{G(Q_j)}(x) - m\bigg|\leq \ve\,m$$
for any $m$ big enough. In particular, choosing $\ve=1/2$ we infer that, for some $m_0=m_0(x)$,
$$\sum_{j= h+1}^m  \frac{\mu(Q_j)}{\mu(G(Q_j))}\,\chi_{G(Q_j)}(x)\geq \frac m2
\quad\mbox{ for any $m\geq m_0$.}$$
Since $\mu(G(Q_j))\geq \delta_0\,\mu(Q_j)$, we get
\begin{equation}\label{eqlaw2}
\sum_{j= h+1}^m \chi_{G(Q_j)}(x)\geq  \frac{\delta_0\,m}2 \quad\mbox{ for $m\geq m_0$.}
\end{equation}

For each $j\geq h$, let $Q_j'\in\DD_\WW(R_0)$ be such that $Q_j=\Pi_{\Sigma_0}(Q_j')$.
Recall that the Key Lemma (we can apply this because $F(x_{Q_j'},A\ell(Q_j'))\geq N_0$)
asserts that
$$F(x_{Q_{j+1}'},A\ell(Q_{j+1}'))\leq \frac12\,F(x_{Q_j'},A\ell(Q_j')) \quad\mbox{ if $x\in G(Q_j)$}$$
and otherwise just ensures that
$$F(x_{Q_{j+1}'},A\ell(Q_{j+1}'))\leq (1+CA^{-1/2})\,F(x_{Q_j'},A\ell(Q_j')) \quad\mbox{ if $x\in Q_j\setminus G(Q_j)$}.$$
These estimates and \rf{eqlaw2} imply that
$$F(x_{Q_{m+1}'},A\ell(Q_{m+1}'))\leq \left(\frac12\right)^{\delta_0m/2} \,(1+CA^{-1/2})^m
\quad\mbox{ for $m\geq m_0$.}$$
However, if $A$ is chosen big enough (recall that $A$ is independent of $\delta_0$ and can be taken
arbitrarily big in the Key Lemma \ref{keylemma}), this implies that
$$F(x_{Q_{m}'},A\ell(Q_{m}'))\to 0\quad\mbox{ as $m\to\infty$,}$$
which cannot happen because $x\not\in  \bigcup_{j\geq h} T_j$, recalling the definition of $T_j$.
This concludes the proof of the claim and of the lemma.
\end{proof}
\vv

The proof of Theorem \ref{teomain} will follow as a straightforward consequence of Lemma \ref{lemliminf}
and the next result of Adolfsson and Escauriaza:

\begin{lemma}\cite[Lemma 0.2]{AE} \label{lemAE}
Let $D\subset\R^n$ be a Lipschitz domain and let $V$ be a relatively open subset of $\partial D$.
Let $v$ be a non-zero function harmonic in $D$ and continuous in $\overline D$ which vanishes identically
in $V$, and whose normal derivative $\partial_\nu v$ vanishes in a subset $E\subset V$ of positive
surface measure. Then, for every point $x\in V$ which is a density point of $E$ (with respect to surface measure),
we have
\begin{equation}\label{eqnodob}
\lim_{r\to 0} \frac{\int_{B(x,r)\cap D} |v|\,dm}{\int_{B(x,6r)\cap D} |v|\,dm} = 0.
\end{equation}
\end{lemma}

Actually, the identity \rf{eqnodob} is not stated explicitly in
Lemma 0.2 from \cite{AE}. Instead, it is said that $v$ vanishes to infinite order in $x$. However, a quick 
inspection of the proof shows that the authors actually prove \rf{eqnodob}, which in turn implies that
$v$ vanishes to infinite order in $x$. 
The lemma above relies on \cite[Lemma 1 and Theorem 1]{AEK}. Though the proof of \cite[Lemma 1]{AEK} is not correct - as explained in \cite[paragraph before Lemma 5]{AEWZ} -  one can replace that lemma either
by \cite[Lemma 2.2]{AE} or by more quantitative arguments involving \cite[Lemma 4]{AEWZ} and well known properties of harmonic functions\footnote{I thank Luis Escauriaza for informing me about this fact.}.  For the reader's convenience I provide an alternative self-contained proof in the Appendix \ref{app3}.

\vv

\begin{proof}[\bf Proof of Theorem \ref{teomain}]
As explained at the beginning of Section \ref{seckey}, it suffices to show that $\partial_\nu u$ cannot vanish
in a subset of positive surface measure of $\Sigma_0 \cap \CC(R_0)$ (since this set contains the ball $B_0$). 

For the sake of contradiction, suppose that $\partial_\nu u$ vanishes in a subset $E\subset
\Sigma_0 \cap \CC(R_0)$ of positive surface measure. By Lemma \ref{lemAE}, for any $x\in \Sigma_0 \cap \CC(R_0)$
which is density point of $E$, 
$$\lim_{r\to 0} \frac{\int_{B(x,6r)} |u|\,dm}{\int_{B(x,r)} |u|\,dm} = \infty.$$
By the subharmonicity of $|u|$, for $r$ small enough,
$$h(x,r/2)^{1/2} = \left(\avint_{\partial B(x,r/2)}|u|^2\,d\sigma\right)^{1/2}\lesssim \avint_{B(x,r)}|u|\,dm.$$
Also, by Cauchy-Schwarz and the fact that $h(x,\cdot)$ is non-decreasing in $r$,
$$\avint_{B(x,6r)} |u|\,dm \leq \left(\avint_{B(x,6r)} |u|^2\,dm\right)^{1/2}\leq
h(x,6r)^{1/2}.$$
Therefore,
$$\liminf_{r\to0} \frac{h(x,6r)^{1/2}}{h(x,r/2)^{1/2}} \gtrsim \liminf_{r\to0}  \frac{\avint_{B(x,6r)} |u|\,dm}{\avint_{B(x,r)} |u|\,dm} = \infty.$$
Consequently,
$$\lim_{r\to0} \frac{h(x,12r)}{h(x,r)}=\infty,$$
which contradicts Lemma \ref{lemliminf}.
\end{proof}
\vv


\appendix

\section{Existence of non-tangential limits for $\nabla u$}

Let $\Omega\subset\R^n$ be a Lipschitz domain, and let $\sigma$ denote the surface measure on
$\partial\Omega$. For $\sigma$-a.e.\ $x$, there exists a tangent hyperplane to $\partial\Omega$ in $x$ and 
the outer unit normal $\nu(x)$ is well defined.
For an aperture parameter $a\in(0,1)$ we consider the one sided inner cone with axis in the direction of $-\nu(x)$ defined by
$$X_a^+(x)=\bigl\{y\in\R^{n}:(x-y)\cdot \nu(x)> a|y-x|\bigr\}.$$
Analogously, we consider the outer cone
$$X_a^-(x)=\bigl\{y\in\R^{n}:(y-x)\cdot \nu(x)> a|y-x|\bigr\}.$$
For a given function $f:\R^n\setminus\partial\Omega\to\R$ and a fixed parameter $a\in(0,1)$, we define the non-tangential limits
\begin{equation*}
f_{+,a}(x) = \lim_{X_a^+(x)\ni y\to x} f(y),\qquad f_{-,a}(x) = \lim_{X_a^-(x)\ni y\to x} f(y),
\end{equation*}
whenever they exist. 

Although the following result is already known (see \cite[Theorem 5.19]{Kenig-Pipher}, I have not
been able to find an easy argument in the literature and thus I provide a detailed proof based on Dahlberg’s theorem on harmonic measure \cite{Dah}.

\begin{theorem}\label{teoap1}
Let $\Omega\subset\R^n$ be a Lipschitz domain, let $B$ be an open ball centered in $\partial\Omega$,
and let $\Sigma = B\cap \partial\Omega$ be a Lipschitz graph.
Let $u$ be a function harmonic in $\Omega$ and continuous in $\overline\Omega$. Suppose that $u$  vanishes in $\Sigma$. Then, for any $a\in(0,1)$, $(\nabla u)_{+,a}$ exists $\sigma$-a.e.\
and belongs to $L^2_{loc}(\sigma|_\Sigma)$. Further, $(\nabla u)_{+,a}$
 has vanishing tangential component. That is,  $(\nabla u)_{+,a}
 = (\partial_\nu u)\,\nu$. 
 Further, assuming that $\Omega\cap B$ is above $\Sigma$, 
$$
\lim_{\ve\to0} \nabla u(\cdot+\ve\, e_n) \to (\partial_\nu u)\,\nu \quad \mbox{ in $L^2_{loc}(\sigma|_\Sigma).$}
$$ 
 Also, in the sense of distributions,
$$(\Delta u)|_B = -\partial_\nu u\,\,\sigma|_\Sigma.$$
\end{theorem}

\begin{proof}
We extend $u$ by $0$ out of $\Omega$ and denote $u^+=\max(u,0)$, $u^-=-\min(u,0)$, so that $u^+$ and $u^-$ are continuous and subharmonic in $B$. This implies that, in the sense of distributions, in $B$, $\Delta u = \Delta u^+ - \Delta u^-$ is a signed Radon measure
supported on $\Sigma$.

First we claim that
\begin{equation}\label{eqclaimap}
(\Delta u)|_B = \rho\,\omega|_\Sigma,
\end{equation}
where $\rho\in L^\infty_{loc}(\Sigma)$ and $\omega$ is the harmonic measure for $\Omega$ with respect to some fix pole $p\in\Omega$. To prove the lemma we may assume $B$ small enough 
so that $\Omega\setminus 2B\neq\varnothing$ and that $p\in\Omega\setminus 2B$.
To prove the claim, let $B'$ be an open ball concentric with $B$ such that 
$\overline {B'}\subset B$. We will show that that there exists some constant $C_2$ depending on $B'$ and $p$ such that for any compact set $K\subset \Sigma$, it holds
\begin{equation}\label{eqclaimap2}
\big|\langle \Delta u,\chi_K\rangle \big|\leq C_2\,\omega(K).
\end{equation}
By duality, this implies \rf{eqclaimap}.

 Given an arbitrary $\ve\in \big(0,\frac12\dist(K,\R^n\setminus B')\big)$,
let $\{Q_i\}_{i\in I}$ be a lattice of cubes which cover $\R^n$, with diameter equal to $\ve/2$. Let 
$\{\vphi_i\}_{i\in I}$ be a partition of unity of $\R^n$, so that each $\vphi_i$ is supported
in $2Q_i$ and $\|\nabla^j\vphi_i\|_\infty\lesssim \ell(Q_i)^{-j}$, for $j=0,1,2$.
Then we have
\begin{equation}\label{eqap2}
\langle \Delta u,\chi_K\rangle = \Big\langle \Delta u,\,\sum_{i\in I'} \vphi_i\Big\rangle - 
\Big\langle \Delta u,\,\sum_{i\in I'} \vphi_i -\chi_K\Big\rangle,
\end{equation}
where $I'$ is the collection of indices $i\in I$ such that $2Q_i\cap K\neq\varnothing$.
Since $(\Delta u)|_B$ is a signed Radon measure,
$$\Big|\Big\langle \Delta u,\,\sum_{i\in I'} \vphi_i -\chi_K\Big\rangle\Big|
\leq \big\langle |\Delta u|,\, \chi_{U_\ve(K)\setminus K}\big\rangle\to 0\quad
\mbox{ as $\ve\to0$,}$$
where $U_\ve(K)$ is the $\ve$-neighborhood of $K$.
Concerning the other term in \rf{eqap2}, we have
$$\Big|\Big\langle \Delta u,\,\sum_{i\in I'} \vphi_i\Big\rangle\Big| \leq \sum_{i\in I'}
\Big|\Big\langle  u,\, \Delta\vphi_i\Big\rangle\Big| \lesssim \sum_{i\in I'}\frac1{\ell(Q_i)^2}\,\int_{2Q_i}
|u|\,dm.$$
Since $|u|$ is subharmonic and continuous in $B$ and it vanishes $B\setminus\Omega$, by the
boundary Harnack principle (see Theorem 5.1 from \cite{JK}, for example), we have
$$|u(x)| \leq C_3\,g(x,p)\quad \mbox{ for all $x\in B'\cap\Omega$,}$$
where $g(\cdot,\cdot)$ is the Green function of $\Omega$ and $C_3$ depends on $u$, $p$, and $B'$, but not on $K$.
Thus, 
$$\Big|\Big\langle \Delta u,\,\sum_{i\in I'} \vphi_i\Big\rangle\Big| \lesssim \sum_{i\in I'}
\frac1{\ell(Q_i)^2}\,\int_{2Q_i} g(x,p)\,dx,$$
with the implicit constant depending on $u$, $p$, and $B'$. By standard estimates for
harmonic measure (see (4.3) and (4.4) from \cite{JK}, for example), we have
$$g(x,p)\lesssim \frac{\omega(4Q_i)}{\ell(Q_i)^{n-2}} \quad \mbox{ for all $x\in 2Q_i\cap \Omega$, $i\in I$.}$$
Therefore,
$$\Big|\Big\langle \Delta u,\,\sum_{i\in I'} \vphi_i\Big\rangle\Big| \lesssim 
\sum_{i\in I'}\omega(4Q_i) \leq \omega(U_{4\ve}(K)).
$$
Letting $\ve\to 0$, we have $\omega(U_{4\ve}(K))\to\omega(K)$ and thus \rf{eqclaimap2} follows, 
which implies the claim \rf{eqclaimap} .

Next recall that by Dahlberg's theorem, harmonic measure on a Lipschitz domain $\Omega$ is a $B_2$ weight with respect to the surface measure $\sigma$. In particular, 
the density function $\frac{d\omega}{d\sigma}$ belongs to $L^2_{loc}(\sigma)$. Therefore,
in the sense of distributions,
$$(\Delta u)|_B = h\,\sigma|_\Sigma,\quad \mbox{ for some $h\in L^2_{loc}(\sigma)$.}$$

Our next objective consists in showing that
$(\nabla u)_{+,a}$ exists $\sigma$-a.e.\ and moreover $(\nabla u)_{+,a} = (\partial_\nu u)\,\nu \in L^2_{loc}(\sigma|_\Sigma)$. To this end, consider an arbitrary open ball $\wt B$ centered in $\Sigma$ such that
$4\wt B\subset B$. Let $\vphi$ be a $C^\infty$ function which equals $1$ on $2\wt B$ and vanishes out of $3\wt B$, and let $v=\vphi\, u$. Observe that
\begin{equation}\label{eqvx}
v = \EE * \Delta(\vphi\,v) = \EE * \big(\vphi\,\Delta u + u\,\Delta\vphi + 2\,\nabla u \cdot \nabla \vphi\big),
\end{equation}
where $\EE$ is the fundamental solution of the Laplacian. Remark also that 
$\nabla u\in L^2_{loc}(B)$, by Caccioppoli's inequality.
 
For a finite Borel measure $\eta$, let $R\eta$ be the $(n-1)$-dimensional Riesz transform of
$\eta$. That is,
$$R\eta(x) = \int \frac{x-y}{|x-y|^n}\,d\eta(y),$$
whenever the integral makes sense.
From the identity \rf{eqvx}, we deduce that, for all $x\not \in\Sigma$,
$$\nabla v(x) = c_n\,\big(R(\vphi\,g\,\sigma|_\Sigma)(x) + R(u\,\Delta\vphi\,m)(x) + 2\,R(\nabla u \cdot \nabla \vphi \,m\big)(x)\big)$$
(recall that $m$ is the Lebesgue measure in $\R^n$). Observe that 
$R(u\,\Delta\vphi\,m)$ and 
$R(\nabla u \cdot \nabla \vphi \,m\big)$ are continuous functions in $\wt B$.
On the other hand, the non-tangential limit $(R(\vphi\,g\,\sigma|_\Sigma))_{\pm,a}(x)$ exists for $\sigma$-a.e.\ $x\in\partial\Omega$, by the classical jump formulas for the Riesz transforms
(see \cite{Tolsa-jumps}, for example). 
Taking also into account that $\nabla v=\nabla u$ in $\wt B$, it follows then that
$(\nabla u)_{\pm,a}(x)$  exists for $\sigma$-a.e.\ $x\in\Sigma\cap \wt B$. By the $L^2(\sigma)$
boundedness of the (principal value) Riesz transform operator $R(\cdot\,\sigma)$ on Lipschitz graphs, we deduce that 
$(\nabla u)_{\pm,a}\in L^2(\sigma|_{\Sigma\cap\wt B})$. 

Since $u\equiv0$ in
$\Omega^c$, it is clear that $(\nabla u)_{-,a}\equiv0$ in $\Sigma\cap\wt B$. As the tangential component of $R(\vphi\,g\,\sigma|_\Sigma)(x)$ is continuous across $\partial\Omega$ for $\sigma$-a.e.\ $x\in\partial\Omega$, again by the jump formulas for the Riesz transforms, we deduce that the tangential component of $(\nabla u)_{+,a}$ coincides with
the tangential component of $(\nabla u)_{-,a}$ $\sigma$-a.e.\ in $\Sigma\cap\wt B$, and thus
$(\nabla u)_{+,a}\equiv0$ in $\Sigma\cap\wt B$, which is equivalent to saying that
$(\nabla u)_{+,a}
 = (\partial_\nu u)\,\nu$ in $\Sigma\cap\wt B$.

It remains to prove that $(\Delta u)|_B = -\partial_\nu u\,\,\sigma|_\Sigma$ in the sense of distributions. To this end, let $\psi$ be a $C^\infty$ function supported in $\wt B$.
Without loss of generality we may assume that $\Sigma\cap \wt B$ is a Lipschitz graph with respect the horizontal axes and that $\Omega\cap \wt B$ lies above $\Sigma\cap\wt B$. For $0<\ve\ll r(\wt B)$, denote $\Sigma_\ve = \Sigma + \ve\,e_n$
and $\Omega_\ve = \Omega + \ve\,e_n$,
where $e_n=(0,\ldots,0,1)$. Then we have
\begin{align}\label{eqrt943}
\langle \Delta u,\,\psi\rangle &= \int u\,\Delta\psi\,dm  = \lim_{\ve\to 0} \int_{\wt B \cap\Omega_\ve} u\,\Delta\psi\,dm \\
& = \lim_{\ve\to 0} \int_{\wt B \cap\partial \Omega_\ve} u\,\partial_\nu\psi\,d\sigma
- \lim_{\ve\to 0} \int_{\wt B \cap\partial \Omega_\ve} \psi\,\partial_\nu u\,d\sigma\nonumber\\
& = 0 -  \int_{\Sigma} \psi\,\partial_\nu u\,d\sigma.\nonumber
\end{align}
The last identity follows from the fact that, in a neighborhood of $\Sigma\cap \wt B$, as $\ve\to0$, $u(\cdot\, +\ve\,e_n)$ converges 
uniformly to $0$ and $\nabla u(\cdot \,+ \ve\,e_n)$ converges  to $(\nabla u)_{+,a}$
in $L^2(\sigma|_{\Sigma\cap \wt B})$ (this is proven by arguments analogous to the ones above for the $\sigma$-a.e.\ existence of the limit $(\nabla u)_{+,a}(x)$ in $\Sigma$). From \rf{eqrt943}, we deduce that $\Delta u = -\partial_\nu u\,\,\sigma|_\Sigma$ in $\wt B$, and thus also in $B$.
\end{proof}

\vv

\section{The Whitney cubes}

In this appendix we prove some of the properties of the Whitney cubes constructed at the
beginning of Section \ref{seckey}.

\begin{lemma}\label{lem-whitney}
Let $\Omega\subsetneq\R^n$ be open. Then there exists a family $\WW$ 
of dyadic cubes with disjoint interiors such that
$$\bigcup_{Q\in\WW} Q = \Omega,$$
and moreover there are
 some constants $\Lambda>20$ and $D_0\geq1$ such the following holds for every $Q\in\WW$:
\begin{itemize}
\item[(i)] $10Q \subset \Omega$ and $\diam(Q)< \frac1{20}\,\dist(Q,\partial\Omega)$;
\item[(ii)] $\Lambda Q \cap \partial\Omega \neq \varnothing$;
\item[(iii)] there are at most $D_0$ cubes $Q'\in\WW$
such that $10Q \cap 10Q' \neq \varnothing$. Further, for such cubes $Q'$, we have $\frac12\ell(Q')\leq \ell(Q)\leq 2\ell(Q')$.
\end{itemize}
\end{lemma}

\begin{proof}
We assume the dyadic cubes to be half open-closed.
Let $c_0\in(0,1/2)$ be some constant to be fixed below. Denote $d(x)  =\dist(x,\partial\Omega)$, and 
let $\WW$ be the family of all dyadic cubes $Q\subset\Omega$ that satisfy
\begin{equation}\label{eqlqdx}
\ell(Q)\leq c_0\,\inf_{x\in Q} d(x)
\end{equation}
and moreover are maximal. It is immediate to check that the cubes from $\WW$ cover
$\Omega$ and are disjoint, because they are maximal.

For all $Q\in\WW$, since $c_0\,d(x)\geq\ell(Q)$ for all $x\in Q$, it follows that $$\dist(Q,\partial
\Omega)\geq c_0^{-1}\ell(Q).$$
Taking $c_0$ small enough, the properties (i) and (ii) follow.

Let $Q,Q'\in\WW$ satisfy $10Q\cap10Q'\neq\varnothing$. Let $Q''$ the dyadic parent of $Q'$, which
is also contained in $\Omega$, by (i). By the definition of $\WW$, there exists $x''\in \overline{Q''}$ such that 
$\ell(Q'')\geq c_0\, d(x'')$. Fix also any $x\in Q$. From the condition $10Q\cap10Q'\neq\varnothing$ it follows that $|x-x''|\leq C\,(\ell(Q)+\ell(Q'))$, where $C$ is some constant depending just on $n$.
Then we have
$$\ell(Q)\leq c_0\,d(x) \leq c_0\,d(x'') + c_0\,|x-x''|\leq 2\,\ell(Q') + c_0\,C(\ell(Q) + \ell(Q')).$$
For $c_0$ small enough we deduce that $\ell(Q)\leq 2.5\,\ell(Q')$, which implies that 
$\ell(Q)\leq 2\,\ell(Q')$ because $\ell(Q)/\ell(Q')=2^k$ for some $k\in\Z$. Reversing the roles of $Q$ and $Q'$ we deduce that $\frac12\ell(Q')\leq \ell(Q)\leq 2\ell(Q')$.
From this property and standard volume considerations it follows easily that there are at most $D_0$ cubes $Q'\in\WW$
such that $10Q \cap 10Q' \neq \varnothing$, with $D_0$ depending just on $n$.
\end{proof}

\vv

\begin{lemma}\label{lemwhit2}
Let $\Omega$ be a Lipschitz domain, and let $\Sigma$, $B_0$, $\Sigma_0$, $H_0$, $\Pi$, and $R_0$ as in 
Section \ref{seckey}.
Also, let $J(R_0)$ be as in \rf{eqjr0} and, for $k\geq1$, let $J_k(R_0)\subset J(R_0)$ be the subfamily of $(n-1)$-dimensional dyadic cubes in $H_0$ with side length equal to $2^{-k}\ell(R_0)$.
Then, for each $Q'\in J_k(R_0)$ there exists some cube $Q\in\WW$ such that $\Pi(Q)=Q'$, $\Pi(Q)\subset \Pi(R_0)$, and such that $Q$ is below $R_0$.
\end{lemma}

\begin{proof}
Let $Q'\in J_k(R_0)$ and denote by $x'$ its center, so that $x'\in H_0\cap\Pi(R_0)$ too.
Let $L_{x'}$ be the line orthogonal to $H_0$ through $x'$. Let $x$ be intersection of $L_{x'}$ with the lower face of $R_0$, and let $x''=L_{x'}\cap \Sigma_0$. 
Let $S$ be the segment $(x,x'')$, which lies on $L_{x'}$.

Consider the sequence of dyadic Whitney cubes $\{R_j\}_{j\geq1}\subset \WW$ which intersect $S$,
so that 
$$S= L_{x'}\cap\bigcup_{j\geq 1} R_j,$$
and assume that the sequence is ordered in such a way that, for all $j\geq0$, $R_j$ and $R_{j+1}$ are neighbors and $R_{j+1}$ is below $R_j$.
The sequence of side lengths $\ell(R_j)$ tends to $0$ as $j\to\infty$ because
$\dist(R_j,\partial\Omega)\to0$ as $j\to\infty$. Also, for any $j\geq0$, $\ell(R_j) / \ell(R_{j+1})$ 
equals $1$, $1/2$, or $2$, by the property (iii) in Lemma \ref{lem-whitney}.
This implies that, for some $j\geq1$, $\ell(R_j)=2^{-k}\ell(R_0)$. Indeed, we claim that
 the cube $R_j$ such that $\ell(R_j) \leq 2^{-k} \ell(R_0)$ and $j$ is minimal does the job. To check this, notice that, by the minimality of $j$, $\ell(R_{j-1}) \geq 2^{-k+1} \ell(R_0)$. So the property (iii) in Lemma \ref{lem-whitney}
implies that $\ell(R_j)\geq  2^{-k} \ell(R_0)$ and the claim follows.

Notice now that $\Pi(R_j)=Q'$, because both $\Pi(R_j)$ and $Q'$ are $(n-1)$-dimensional dyadic cubes in $H_0$ with side length $2^{-k}\ell(R_0)$ and both contain $x'$.
\end{proof}
\vv


\section{An alternative proof of Lemma \ref{lemAE}}\label{app3}

We assume that we are under the assumptions of Lemma \ref{lemAE}. So given a Lipschitz domain
 $D\subset\R^n$ and a relatively open subset $V$ of $\partial D$,  we consider 
  a non-zero function $v$ which is harmonic in $D$ and continuous in $\overline D$, vanishing identically
in $V$, and whose normal derivative $\partial_\nu v$ also vanishes in a subset $E\subset V$ of positive
surface measure. We have to show that, for every point $x\in V$ which is a density point of $E$ (with respect to surface measure),
it holds
\begin{equation}\label{eqnodob'}
\lim_{r\to 0} \frac{\int_{B(x,r)\cap D} |v|\,dm}{\int_{B(x,6r)\cap D} |v|\,dm} = 0.
\end{equation}
To this end, for such point $x$, given $\ve\in (0,1)$, let $r_0>0$ be small enough so that $B(x,r_0)\subset D$, 
$B(x,r_0)\cap \partial D\subset V$, and
$$\sigma(\partial D\cap B(x,r)\setminus E)\leq \ve\,\sigma(\partial D\cap B(x,r))\quad \mbox{ for all $0<r\leq r_0$.}$$
We fix $r\in (0,r_0/3)$.
Without loss of generality, we assume that $x=0$ and we denote $B_r=B(0,r)$, We also assume
that $\partial D\cap B_{2r}$ is a Lipschitz graph with respect to the horizontal axes, and that $D$ is above the graph. As usual, we understand that $v$ has been extended by $0$ in $D^c$.
As shown in Theorem \ref{teoap1},
$$(\Delta v)|_{B_r} = -\partial_\nu v\,\,\sigma|_{\partial D \cap B_r} =:\mu.$$
Thus, 
$g:=v - \EE *\mu$ is harmonic in $B_r.$ 

We intend to apply the three ball inequality to the function $g$. In order to do
this, first we need to estimate the total variation of the signed measure $\mu$.  We apply the Rellich-Necas identity 
$$2(\beta \cdot \nabla v)\,\Delta v = 2\,{\rm div}((\beta\cdot\nabla v)\,\nabla v) -
{\rm div}(\beta\,|\nabla v|^2) + |\nabla v|^2\,{\rm div}\beta  - 2 \sum_{i,j}\partial_i\beta_j\,
\partial_i v\,\partial_j v,$$
with a vector field $\beta = \vphi\,e_n$, where $\vphi$ is a smooth function supported on
$B_{\frac32r}$ and identically $1$ on $B_r$. Integrating the above identity in $B_{2r}\cap D$ with 
respect to Lebesgue measure and applying the divergence theorem, we obtain
\begin{align*}
0= & \int_{\partial D\cap B_{2r}} \big[2\,\vphi(y)\, \partial_n v(y)\,\partial_\nu v(y) - 
\vphi(y) \,|\nabla v(y)|^2 \,(e_n\cdot \nu(y)) \big]\,d\sigma(y)\\
&\quad + \int_{B_{2r}}\partial_n\vphi\,|\nabla v|^2\,dy - 2 \int_{B_{2r}}\sum_{i}\partial_i\vphi_n\,
\partial_i v\,\partial_n v\,dy.
\end{align*}
In fact, to be more precise, since $v$ need not be smooth up to $\partial D$, first we apply
the divergence theorem in the domain $B_{2r}\cap (\delta e_n + D)$, with $\delta>0$, and then we let
$\delta\to 0$ as in the proof of Lemma \ref{lemderiv}.
Taking into account that $\nabla v= (\partial_\nu v)\nu$ on $\partial D\cap B_{2r}$, we get
$$\int_{\partial D\cap B_{2r}}
\vphi(y) \,|\partial_\nu v(y)|^2 \,(e_n\cdot \nu(y)) \,d\sigma(y) =
 -\int_{B_{2r}}\!\partial_n\vphi\,|\nabla v|^2\,dy + 2 \int_{B_{2r}}\sum_{i}\partial_i\vphi_n\,
\partial_i v\,\partial_n v\,dy.$$
Thus, recalling that $\chi_{B_r}\leq \vphi\leq \chi_{B_{\frac32 r}}$ and applying Caccioppoli's
inequality,
$$\int_{\partial D\cap B_{r}}
|\partial_\nu v(y)|^2 \,d\sigma(y) \lesssim \frac1r \int_{B_{\frac32r}}|\nabla v|^2\,dy
\lesssim \frac1{r^3} \int_{B_{2r}}|v|^2\,dy
.$$
Now, by Cauchy-Schwarz, we derive
\begin{align*}
\|\mu\|  = \int_{\partial D\cap B_{r}}
|\partial_\nu v(y)| \,d\sigma(y) & \leq 
\left(\int_{\partial D\cap B_{r}}
|\partial_\nu v(y)|^2 \,d\sigma(y)\right)^{1/2}
\sigma(B_r\cap \partial D\setminus E)^{1/2}\\
& \lesssim \ve^{\frac12}\, r^{\frac n2 -2}  \left(\int_{B_{2r}}|v|^2\,dy\right)^{1/2}.
\end{align*}

Let $x'=-\frac {r}{10}\,e_n$. Then there is some $c_2>0$ depending just on
 the Lipschitz character of $D$ such that
 $\overline{B(x',2c_2r)}\subset B_r\setminus \overline{D}$. We denote $r'=c_2r$.
  Observe now that for any $s\in[r',r]$,
\begin{equation}\label{eqws744}
\int_{\partial B(x',s)} |\EE * \mu| \,d\sigma \lesssim \int \!\!\int_{\partial B(x',s)}
\frac1{|y-z|^{n-2}}\,d\sigma(y)\,d|\mu|(z) \lesssim \|\mu\|\,s.
\end{equation}
Thus, the function\footnote{In the case $n=2$ it is better to choose 
$w= \EE * \mu - \frac1{2\pi}\, \mu(\R^2)\,\log r = \frac1{2\pi}\,\log\frac{|x|}r * \mu$, and then
\rf{eqws744} and \rf{eqws745} follow analogously.}
 $w:=\EE*\mu$ satisfies
\begin{equation}\label{eqws745}
\avint_{\partial B(x',s)}|w|\,d\sigma \lesssim \frac1{r^{n-2}}\,\|\mu\|
\lesssim \ve^{1/2}\, \frac1{r^{n/2}} \left(\int_{B_{2r}}|v|^2\,dy\right)^{1/2}
\approx  \ve^{1/2}\,  \left(\avint_{B_{2r}}|v|^2\,dy\right)^{1/2}
,
\end{equation}
with the implicit constants, from now on, possibly depending on $c_2$ and thus on the Lipschitz character of $D$.

Since $g=w$ in $\overline{B(x',2r')}$, and $g$ is harmonic in $B(x',2r')$ and continuous in
$\overline{B(x',2r')}$, we deduce that, for all $z\in 
\overline{B(x',r')}$, 
$$|g(z)|\lesssim \avint_{\partial B(x',2r')}|g|\,d\sigma 
\lesssim \ve^{1/2} \left(\avint_{B_{2r}}|v|^2\,dy\right)^{1/2}.
$$
Hence,
$$h_g(x',r') := \avint_{\partial B(x',r')}|g|^2\,d\sigma \lesssim \ve\avint_{B_{2r}}|v|^2\,dy.
$$

Next we estimate $h_g(x',\frac34 r)$. We write
$$\avint_{\partial B(x',\frac45r)}|g|\,d\sigma \leq \avint_{\partial B(x',\frac45r)}|v|\,d\sigma
+ \avint_{\partial B(x',\frac45r)}|w|\,d\sigma.$$
Since $|v|$ is subharmonic in $B_{2r}$ and continuous in $\overline{B_{2r}}$, we have
$$|v(z)| \lesssim \avint_{B_{2r}} |v|\,dy\quad \mbox{ for all $y\in B_r$.}$$
Therefore,
$$\avint_{\partial B(x',\frac45r)}|g|\,d\sigma \lesssim \avint_{B_{2r}} |v|\,dy + \ve^{1/2} \left(\avint_{B_{2r}}|v|^2\,dy\right)^{1/2} \lesssim \left(\avint_{B_{2r}}|v|^2\,dy\right)^{1/2}.$$
Since $g$ is harmonic in $B(x',\frac45 r)$ and continuous in
$\overline{B(x',\frac45r)}$, we have, for all $z\in 
\overline{B(x',\frac34 r)}$,
$$|g(z)|\lesssim
\avint_{\partial B(x',\frac45r)}|g|\,d\sigma \lesssim \left(\avint_{B_{2r}}|v|^2\,dy\right)^{1/2},$$
and so 
$$h_g(x',\tfrac34r) \lesssim
\avint_{B_{2r}}|v|^2\,dy.$$

Recall now that, by the three ball inequality, given $\alpha\in (0,1)$ and  $r_2>r_1>0$ such that
$\overline{B(x',r_2)}\subset B_r$, it holds
$$h_g(x',r_1^\alpha r_2^{1-\alpha}) \leq h_g(x',r_1)^\alpha \,h_g(x',r_2)^{1-\alpha}.$$
In fact, this inequality is an immediate consequence of the convexity of the function $\log h_g(x',e^t)$,
proven in Lemma \ref{lemderiv} in a more general situation. Applying this inequality with
$r_1=r'$, $r_2=\frac34r$, and $\alpha$ such that
$$(r')^\alpha \,(\tfrac34r)^{1-\alpha} = \tfrac23 r,$$
i.e., 
$$\alpha = \frac{\log\frac98}{\log\frac 3{4c_2}},$$
we infer that
$$h_g(x',\tfrac23r) \lesssim
\ve^\alpha \avint_{B_{2r}}|v|^2\,dy.$$
Hence, using Cauchy-Schwarz and again \rf{eqws745},
\begin{align*}
\avint_{\partial B(x',\frac23r)} |v| \,d\sigma & \leq 
\avint_{\partial B(x',\frac23r)} |g| \,d\sigma + 
\avint_{\partial B(x',\frac23r)} |w| \,d\sigma \\
& \lesssim h_g(x',\tfrac23r)^{1/2} +
 \ve^{1/2}\,  \left(\avint_{B_{2r}}|v|^2\,dy\right)^{1/2} \lesssim \ve^{\alpha/2}\,  \left(\avint_{B_{2r}}|v|^2\,dy\right)^{1/2}.
 \end{align*}
Since $|v|$ is subharmonic in $B_{3r}$ and continuous in $\overline{B_{3r}}$ and 
$\overline{B_{\frac12 r}}\subset B(x',\frac23 r)$, we get
$$\avint_{B_{\frac12 r}} |v|\,dy\lesssim \avint_{\partial B(x',\frac23r)} |v| \,d\sigma 
\lesssim \ve^{\alpha/2}\,  \left(\avint_{B_{2r}}|v|^2\,dy\right)^{1/2}
\lesssim \ve^{\alpha/2}\,  \avint_{B_{3r}}|v|\,dy.$$
So we have shown that, for any $\ve>0$, if $r$ is small enough
$$\frac{\avint_{B_{\frac12 r}} |v|\,dy}{ \avint_{B_{3r}}|v|\,dy}\lesssim \ve^{\alpha/2},$$
which proves the lemma.
\fiproof


\vvv


\begin{thebibliography}{NTWV}

\bibitem[AE]{AE} V. Adolfsson and L. Escauriaza. {\em $C^{1,\alpha}$ domains and unique continuation at the boundary.} Comm. Pure Appl. Math. 50 (1997), no. 10, 935--969. 

\bibitem[AEK]{AEK} V. Adolfsson, L. Escauriaza, and C. Kenig. {\em Convex domains and unique continuation at the boundary.} Rev. Mat. Iberoamericana 11 (1995), no. 3, 513--525. 

\bibitem[AEWZ]{AEWZ} J. Apraiz, L. Escauriaza, G. Wang, and C. Zhang. {\em Observability inequalities and measurable sets.} J. Eur. Math. Soc. (JEMS) 16 (2014), no. 11, 2433--2475.

\bibitem[ARRV]{ARRV} G. Alessandrini, L. Rondi, E. Rosset, and S. Vessella. {\em The stability for the Cauchy problem for elliptic equations.} Inverse Problems 25 (2009), no. 12, 123004, 47 pp. 

\bibitem[BW]{BW} J. Bourgain and T. Wolff. {\em
A remark on gradients of harmonic functions in dimension $\geq 3$.}
Colloq. Math. 60/61 (1990), no. 1, 253--260. 

\bibitem[Dah]{Dah} B. Dahlberg. {\em  
On estimates for harmonic measure.}
{Arch. Rat. Mech. Analysis} { 65} (1977), 272--288.

\bibitem[Et]{Etemadi} N. Etemadi. {On the laws of large numbers for nonnegative random variables.} J. Multivariate Anal. 13 (1983), no. 1, 187--193.

\bibitem[JK]{JK} D.S. Jerison and C.E. Kenig. {\em Boundary behavior of harmonic functions in nontangentially accessible domains.} Adv. Math. 46 (1982), no. 1, 80--147.

\bibitem[KN]{KN} I. Kukavica and K. Nystr\"om. {\em Unique continuation on the boundary for Dini domains.} Proc. Amer. Math. Soc. 126 (1998), no. 2, 441--446.  

\bibitem[KP]{Kenig-Pipher} C.E.\ Kenig and J.\ Pipher. {\em
The Neumann problem for elliptic equations with non-smooth coefficients.} 
Invent. Math. 113, No. 3 (1993), 447--509.

\bibitem[KZ]{Kenig-Zhao} C.E.\ Kenig and Z.\ Zhao. {\em Boundary unique continuation on $C^1$-Dini domains and the size of the singular set.} Preprint arXiv:2102.07281 (2021).

\bibitem[Lin]{Lin} F.-H. Lin. {\em Nodal sets of solutions of elliptic and parabolic equations.} Comm. Pure Appl. Math. 44 (1991), 287--308.

\bibitem[Lo1]{Logunov1} A. Logunov. {\em Nodal sets of Laplace eigenfunctions: polynomial upper estimates of the Hausdorff measure.} Ann. of Math. (2) 187 (2018), no. 1, 221--239. 

\bibitem[Lo2]{Logunov2} A. Logunov. {\em Nodal sets of Laplace eigenfunctions: proof of Nadirashvili's conjecture and of the lower bound in Yau's conjecture.} Ann. of Math. (2) 187 (2018), no. 1, 241--262. 

\bibitem[LM]{LM} A. Logunov and E. Malinnikova. {\em Nodal sets of Laplace eigenfunctions: estimates of the Hausdorff
measure in dimension two and three.} Oper. Theory Adv. Appl. 261 (2018), 333--344. 

\bibitem[Mc]{MC} S. McCurdy. {\em Unique continuation on convex domains.} Preprint 
arXiv:1907.02640 (2019).

\bibitem[NV]{NV}  A. Naber and D. Valtorta. {\em Rectifiable-Reifenberg and the regularity of stationary and minimizing harmonic maps.} Ann. of Math. (2) 185 (2017), no. 1, 131--227. 

\bibitem[To]{Tolsa-jumps} X. Tolsa. {\em Jump formulas for singular integrals and layer potentials on rectifiable sets.} Proc. Amer. Math. Soc. 148(11) (2020), 4755--4767.

\bibitem[Wa]{Wang} W.S. Wang. {\em A remark on gradients of harmonic functions.} Rev. Mat. Iberoamericana, 11(2) (1995), 227--245.

\end{thebibliography}
\end{document}